\documentclass[11pt]{amsart}
\usepackage{url}
\usepackage{hyperref}
\usepackage{amssymb}
\usepackage{url}
\usepackage{enumerate}
\usepackage{fancyhdr}
\newtheorem{theorem}{Theorem}[section]
\newtheorem{lemma}[theorem]{Lemma}
\newtheorem{corollary}[theorem]{Corollary}
\newtheorem{conjecture}[theorem]{Conjecture}

\theoremstyle{definition}
\newtheorem{definition}[theorem]{Definition}

\theoremstyle{remark}
\newtheorem{remark}[theorem]{Remark}
\newtheorem{remarks}[theorem]{Remarks}
\numberwithin{equation}{section}

\let\ft=\footnotesize
\def\lien{\mathrel{\mkern-4mu}}
\def\too{\relbar\lien\rightarrow}
\def\tooo{\relbar\lien\relbar\lien\too}

\def\Q{\mathbb{Q}}
\def\Z{\mathbb{Z}}

\def\Frac#1#2{\hbox{\footnotesize $\displaystyle \frac{#1}{#2}$}}
\def\plus{\displaystyle\mathop{\raise 1.0pt \hbox{$\bigoplus $}}\limits}
\def\prd{\displaystyle\mathop{\raise 2.0pt \hbox{$\prod$}}\limits}
\def\sm{\displaystyle\mathop{\raise 2.0pt \hbox{$\sum$}}\limits}
\let\ds=\displaystyle
\let\wt=\widetilde
\let\ov=\overline
\def\Cl{{\mathcal C}\hskip-2pt{\ell}}
\def\cl{c\hskip-1pt{\ell}}

\def\order{\raise1.5pt \hbox{${\scriptscriptstyle \#}$}}
\begin{document}
\markboth{Georges Gras}
{Annihilation of ${\rm tor}_{\Z_p}({\mathcal G}_{K,S}^{\rm ab})$}

\title[Annihilation of ${\rm \bf tor}_{\Z_p}^{}({\mathcal G}_{K,S}^{\rm ab})$]
{Annihilation of ${\rm \bf tor}_{\Z_p}^{}({\mathcal G}_{K,S}^{\rm ab})$ \\ 
for real abelian extensions $K/\Q$}

\author{Georges Gras}
\address{Villa la Gardette, Chemin Ch\^ateau Gagni\`ere 
 F--38520 Le Bourg d'Oisans, France,
{\rm \url{https://www.researchgate.net/profile/Georges_Gras}}}
\email{g.mn.gras@wanadoo.fr}


\date{July 5, 2018; September 18, 2018}

\begin{abstract}
Let $K$ be a real abelian extension of $\Q$.
Let $p$ be a prime number, $S$ the set of $p$-places of $K$ and 
${\mathcal G}_{K,S}$ the Galois group of the maximal 
$S \cup \{\infty\}$-ramified pro-$p$-extension of $K$ (i.e., 
unramified outside $p$ and $\infty$). 
We revisit the problem of annihilation of 
the $p$-torsion group ${\mathcal T}_K := 
{\rm tor}_{\Z_p} \big({\mathcal G}_{K,S}^{\rm ab}\big)$
initiated by us and Oriat then systematized in our paper on the construction 
of $p$-adic $L$-functions in which we obtained a 
canonical ideal annihilator of ${\mathcal T}_K$ in full generality
(1978--1981).
Afterwards (1992--2014) some annihilators, using cyclotomic units,
were proposed by Solomon, Belliard--Nguyen Quang Do, 
Nguyen Quang Do--Nicolas, All, Belliard--Martin.
In this text, we improve our original papers and show that, in general, 
the Solomon elements are not optimal and/or partly degenerated.
We obtain, whatever $K$ and $p$, an universal non-degenerated 
annihilator in terms of $p$-adic logarithms of cyclotomic numbers related to
$L_p$-functions at $s=1$ of {\it primitive characters of $K$} (Theorem \ref{thmp}). 
Some computations are given with PARI programs; the case $p=2$ is 
analyzed and illustrated in degrees $2$, $3$, $4$ to test a conjecture.
\end{abstract}

\maketitle

\tableofcontents

\section{Introduction}
Let $K/\Q$ be a real abelian extension of Galois group $G_K$.
Let $p$ be a prime number, $S$ the set of $p$-places of $K$, and 
${\mathcal G}_{K,S}$ the Galois group of the maximal $S$-ramified 
in the ordinary sense (i.e., unramified outside $p$ and $\infty$,
whence totally real if $p=2$) pro-$p$-extension of $K$. 
 
\smallskip
We revisit the classical problem of annihilation of the so-called $\Z_p[G_K]$-module 
${\mathcal T}_K := {\rm tor}_{\Z_p} \big({\mathcal G}_{K,S}^{\rm ab}\big)$, 
as dual of ${\rm H}^2({\mathcal G}_{K,S}, \Z_p(0))$. 
This was initiated by us \cite{Gr5} (1979) and improved by Oriat \cite{O} (1981). 
Then in our paper  \cite{Gr6} (1978/79) on the construction of $p$-adic $L$-functions 
(via an ``arithmetic Mellin transform'' from the ``Spiegel involution'' of 
suitable Stickelberger elements) we obtained incidentally a canonical ideal 
annihilator ${\mathcal A}_K$ of ${\mathcal T}_K$ in full generality, but
our purpose, contrary to the present work, was the semi-simple case
with $p$-adic characters and the annihilation of the isotopic components; 
this aspect has then been outdated by the
``principal theorems'' of Ribet--Mazur--Wiles--Kolyvagin--Greither
(refer for instance to the bibliography of \cite{Gre}), and many other contributions.

\smallskip
Afterwards some annihilators, using cyclotomic units,
were proposed by Solomon  \cite{Sol1} (1992), Belliard--Nguyen Quang 
Do  \cite{B-N} (2005), Nguyen Quang Do--Nicolas  \cite{N-N} (2011), 
All  \cite{All1} (2013), Belliard--Martin  \cite{B-M} (2014), 
using techniques of Sinnott, Rubin, Thaine, Coleman, from Iwasawa's theory.

\smallskip
In this text, we translate into english some parts of the above 
1978--1981's papers, written in french 
with tedious classical techniques, then we show that, in general, the Solomon 
elements $\Psi_K$ are often degenerated regarding the annihilator 
${\mathcal A}_K$, even for cyclic fields, and explain the origin of this 
gap due to trivialization of some Euler factors.

\smallskip
We obtain, whatever $K$ and $p$ (Theorem \ref{thmp}), 
an universal non-degenerated annihilator ${\mathcal A}_K$, in terms 
of $p$-adic logarithms of cyclotomic numbers, perhaps the best possible 
regarding these classical methods, but probably too general to cover
all the possible Galois structures of ${\mathcal T}_K$, 
which raises the question of the existence of a better theorem than 
Stickelberger's one. 

\smallskip
Indeed, if the semi-simple case is now 
completely solved, the non-semi-simple case is far to be known.
Numerical experiments show in this case
that the results are far to give the precise Galois structure
of ${\mathcal T}_K$ (e.g., in direction of its Fitting ideal), moreover, it seems 
to us that many (all ?) papers are based on the classical reasoning with 
Kummer's theory and Leopoldt's Spiegel involution applied to
Stickelberger's elements, even translated into
Iwasawa's theory, without practical analysis of the results 
(e.g., with extensive numerical illustrations). So, there is some
difficulties to compare these various contributions.  

\smallskip
Thus, we perform some computations given with PARI programs \cite{P}
to analyse the quality of such annihilators, which is in general not 
addressed by papers dealing with Iwasawa's theory. 
We consider in a large part the case $p=2$, illustrated
in degrees $2$, $3$, $4$ to test the Conjecture \ref{conj}.

\section{Notations and reminders on $p$-ramification theory}\label{section1} 
Let $K$ be a real abelian number field of degree $d$, of Galois group $G_K$,
and let $p \geq 2$ be a prime number; we denote by $S$ the set 
of prime ideals of $K$ dividing $p$. 
Let ${\mathcal G}_{K,S}$ be the Galois group of the maximal 
$S \cup \{\infty\}$-ramified pro-$p$-extension of $K$ and 
let $H_K^{\rm pr}$ be the maximal abelian $S \cup \{\infty\}$-ramified 
pro-$p$-extension of $K$. 
To simplify, we put ${\mathcal G}_{K,S}^{\rm ab} =: 
{\mathcal G}_{K}$ and (e.g., \cite[Chapter III, \S\,(c)]{Gr1}):
$${\mathcal T}_K :=  {\rm  tor}^{}_{\Z_p}({\mathcal G}_K)
= {\rm Gal}(H_K^{\rm pr}/ K_\infty) $$

where $K_\infty = K\Q_\infty$ is the cyclotomic $\Z_p$-extension of $K$; so:
$${\mathcal G}_K \simeq \Z_p \plus {\mathcal T}_K$$ 
since, in the abelian case, Leopoldt's conjecture is true. 

\smallskip
We denote by $F$ an extension of $K$ such that $H_K^{\rm pr}$ 
is the direct compositum of $K_\infty$ and $F$  over $K$, then by
$\Cl_K^\infty$ the subgroup of the $p$-class group $\Cl_K$ corresponding, 
by class field theory, to ${\rm Gal}(H_K/ K_\infty \cap H_K)$, 
where $H_K$ is the $p$-Hilbert class field. We have (where
$\sim$ means ``equality up to a $p$-adic unit''):
\begin{equation}\label{clinfty}
\order \Cl_K^\infty \sim \Frac{\order \Cl_K}
{[K_\infty \cap H_K : K]} \sim \order \Cl_K \cdot 
\Frac{[K \cap \Q_\infty : \Q]}{e_p}\cdot \Frac{2}
{\order (\langle -1 \rangle \cap  {\rm N}_{K/\Q}(U_K))}, 
\end{equation}

where $e_p$ is the ramification index of $p$ in $K/\Q$ 
\cite[Theorem III.2.6.4]{Gr1}, and $U_K$ is defined as follows:

\smallskip
For each ${\mathfrak p} \mid p$, let $K_{\mathfrak p}$ be the 
${\mathfrak p}$-completion of $K$ and $\ov {\mathfrak p}$ the 
corresponding prime ideal of the ring of integers of $K_{\mathfrak p}$; then let:

\smallskip
\centerline{$U_K := \Big \{u \in \plus_{{\mathfrak p} \mid p}K_{\mathfrak p}^\times, \ \,
u = 1+x, \  x \in \plus_{{\mathfrak p} \mid p} \ov {\mathfrak p} \Big\}\, \ \ \& \ \ \,
W_K := {\rm tor}_{\Z_p}^{}(U_K)$,} 

\smallskip\noindent
the $\Z_p$-module (of $\Z_p$-rank $d = [K : \Q]$) of principal local units at $p$ 
and its  torsion subgroup, respectively; by class field theory this gives 
in the diagram: 
$${\rm Gal}(H_K^{\rm pr}/H_K) \simeq U_K/\ov E_K \ \ \& \ \ 
{\rm Gal}(H_K^{\rm pr}/K_\infty H_K) \simeq {\rm tor}_{\Z_p}^{} \big(U_K \big / \ov E_K \big), $$ 
where $\ov E_K$ is the closure of the group $E_K$ of $p$-principal global units of $K$ 
(i.e., units $\varepsilon \equiv 1 \! \pmod{ \prod_{{\mathfrak p} \mid p} {\mathfrak p}}$):
\unitlength=1.0cm 
$$\vbox{\hbox{\hspace{-1.2cm} 
 \begin{picture}(11.5,5.6)
\put(6.8,4.50){\line(1,0){4.0}}
\put(3.85,4.50){\line(1,0){1.4}}
\put(4.4,2.50){\line(1,0){1.05}}
\put(3.8,0.5){\line(1,0){7.0}}
\put(10.9,0.4){$F$}

\bezier{350}(3.8,4.8)(7.6,5.8)(11.0,4.8)
\put(6.2,5.45){\footnotesize${\mathcal T}_K=
{\rm tor}_{\Z_p}({\mathcal G}_K)$}

\put(3.50,2.9){\line(0,1){1.25}}
\put(3.50,0.9){\line(0,1){1.25}}
\put(5.75,2.9){\line(0,1){1.25}}

\put(11.1,0.8){\line(0,1){3.45}}%

\bezier{300}(3.9,0.65)(5.1,0.8)(5.6,2.3)
\put(5.25,1.3){\footnotesize$\simeq \! \Cl_K$}
\put(4.2,4.15){\footnotesize$\order \Cl_K^\infty$}

\bezier{300}(6.3,2.5)(8.5,2.6)(10.8,4.3)
\put(7.5,2.5){\footnotesize$\simeq \! U_K/\ov E_K$}

\bezier{500}(3.9,0.6)(8.5,0.8)(10.95,4.3)
\put(8.7,1.8){\footnotesize${\mathcal G}_K$}

\put(10.85,4.4){$H_K^{\rm pr}$}
\put(5.4,4.4){$K_\infty\! H_K$}
\put(7.6,4.14){\footnotesize$\order{\mathcal R}_K
 \cdot \order {\mathcal W}_K$}
\put(3.3,4.4){$K_\infty$}
\put(5.55,2.4){$H_K$}
\put(2.8,2.4){$K_\infty \!\cap \! H_K$}
\put(3.3,0.40){$K$}

\put(1.7,0.5){\line(1,0){1.5}}
\put(1.3,0.4){$\Q$}
\put(2.3,0.65){\footnotesize $G_K$}
\end{picture}   }} $$
\unitlength=1.0cm

For any field $k$, let $\mu^{}_k$ be the group of roots of unity of $k$
of $p$-power order. Then $W_K = 
\plus_{{\mathfrak p} \mid p} \mu^{}_{K_{\mathfrak p}}$.
We have the following exact sequence defining 
${\mathcal W}_K$ and ${\mathcal R}_K$ via the $p$-adic logarithm ${\rm log}$
(\cite[Lemma III.4.2.4]{Gr1} or \cite[Lemma 3.1 \& \S\,5]{Gr2}):
\begin{equation}\label{exseq}
\begin{aligned}
1\to  {\mathcal W}_K  :=  W_K /\mu^{}_K &  \tooo 
 {\rm tor}_{\Z_p}^{} \big(U_K \big / \ov E_K \big) \\
 &\mathop {\tooo}^{{\rm log}}  {\rm tor}_{\Z_p}^{}\big({\rm log}\big 
(U_K \big) \big / {\rm log} (\ov E_K) \big)=: {\mathcal R}_K \to 0. 
\end{aligned}
\end{equation}

The group ${\mathcal R}_K$ is called the {\it normalized $p$-adic 
regulator of $K$} and makes sense for any number field (see the above
references in \cite{Gr2} for more details and the main properties of 
these invariants).

\smallskip
It is clear that the annihilation of ${\mathcal T}_K$ mainely
concerns the group ${\mathcal R}_K$ since the $p$-class group
is in general trivial (and so for $p$ large enough) and because 
the regulator may be non-trivial with large valuations and
unpredictible $p$ (see \cite{Gr4} for some conjectures and 
\cite{Gr3} giving programs of fast computation of 
the {\it group structure of ${\mathcal T}_K$
for any number field} given by means of polynomials).

\begin{definition}
A field $K$ is said to be $p$-rational if 
the Leopoldt conjecture is satisfied for $p$ in $K$ and if
the torsion group ${\mathcal T}_K$ is trivial (\cite[Section III, \S\,2]{Gr7}, 
then \cite[Theorem IV.3.5]{Gr1}, \cite{Gr3}, and bibliographies
for the history and properties of $p$-rationality).
\end{definition}

This has deep consequences in Galois theory over $K$ since 
${\mathcal T}_K$ is the dual of
${\rm H}^2({\mathcal G}_{K,S}, \Z_p(0))$ \cite{Ng1}.

\section{Kummer theory  and Spiegel involution} 
\subsection{Kummer theory} 
We denote by $\Q_n$, $n \geq 0$, the $n$th stage in $\Q_\infty$ 
so that $[\Q_n : \Q] = p^n$.
Let $n_0\geq 0$ be defined by $K \cap \Q_\infty=:\Q_{n_0}$.

\smallskip
Let $n \geq n_0$.
We denote by $K_n$ the compositum $K \Q_n$ and by
$F_n$ the compositum $F K_n = F \Q_n$ (in other words, 
$K=K_{n_0}$, $F=F_{n_0}$).
Then we have the {\it group isomorphism}
${\rm Gal}(F_n/K_n) \simeq {\mathcal T}_K$ for all $n \geq n_0$.

\smallskip
Put $q=p$ (resp. $4$) if $p \ne 2$ (resp. $p=2$).
Let $L=K(\mu_q)$ and $M=F(\mu_q)$;
then put $L_n:=L K_n$ for all $n \geq n_0$.

\smallskip
Let $M_n:= F_n (\mu_q)$ (whence $L=L_{n_0}$, $M=M_{n_0}$).
For $p\ne 2$, the degrees $[L_n : K_n] = [M_n : F_n]$ are 
equal to a divisor $\delta$ of $p-1$ independent of $n \geq n_0$
($\delta$ is even since $K$ is real). For $p=2$, $\delta=2$. 
In any case, one has, for $n \geq n_0$:
$$L_n = K(\mu_{q p^{n}}^{}). $$
All this is summarized by the following diagram:

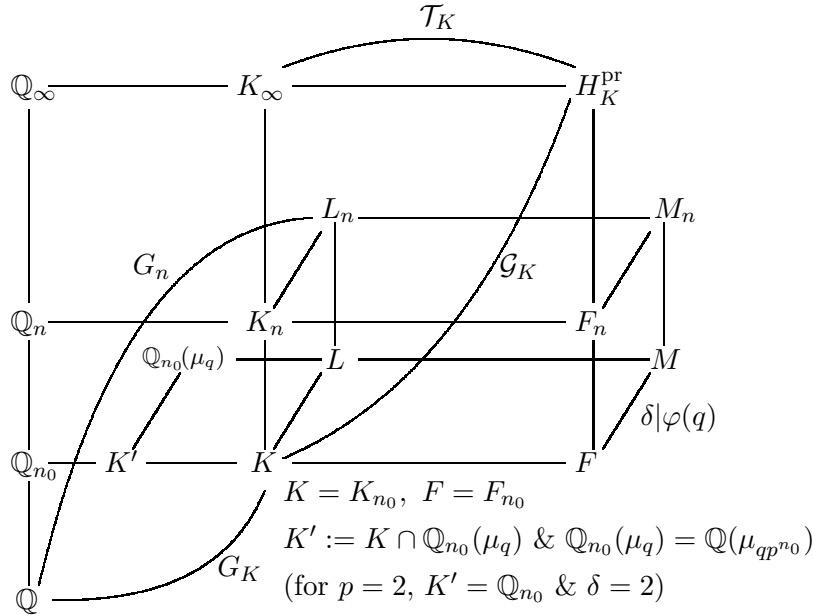
\unitlength=1.25cm
$$\vbox{\hbox{\hspace{-1.0cm} 
 \begin{picture}(9.5,6.15)
\put(4.3,5.5){\line(1,0){2.9}}
\put(1.7,5.5){\line(1,0){1.9}}
\put(4.3,3.0){\line(1,0){2.9}}
\put(1.7,3.0){\line(1,0){1.9}}
\put(4.3,1.5){\line(1,0){2.9}}
\put(1.7,1.5){\line(1,0){0.5}}
\put(2.75,1.5){\line(1,0){0.95}}
\put(3.7,2.6){\line(1,0){0.9}}
\put(2.7,2.55){\ft$\Q_{n_0}\!(\mu_q)$}
\put(2.3,1.4){$K'$}
\put(5.0,2.6){\line(1,0){3.1}}%
\put(5.0,4.1){\line(1,0){3.1}}%
\put(1.5,3.2){\line(0,1){2.05}}
\put(1.5,1.65){\line(0,1){1.15}}
\put(1.5,0.2){\line(0,1){1.1}}
\put(4.0,3.2){\line(0,1){2.05}}
\put(4.0,1.65){\line(0,1){1.15}}
\put(8.25,2.75){\line(0,1){1.3}}
\put(4.75,2.75){\line(0,1){1.3}}
\put(7.5,3.2){\line(0,1){2.05}}
\put(7.5,1.65){\line(0,1){1.15}}
\put(7.3,5.4){$H_K^{\rm pr}$}
\put(3.7,5.4){$K_\infty$}
\put(1.3,5.4){$\Q_\infty$}
\put(1.3,2.9){$\Q_n$}
\put(3.8,2.9){$K_n$}
\put(7.3,2.9){$F_n$}
\put(4.6,4.1){$L_n$}
\put(8.14,4.1){$M_n$}
\bezier{350}(4.1,3.15)(4.35,3.55)(4.6,3.95)
\bezier{350}(7.6,3.15)(7.85,3.55)(8.1,3.95)
\put(1.3,1.4){$\Q_{n_0}$}
\put(3.85,1.4){$K$}
\put(7.3,1.4){$F$}
\put(4.65,2.5){$L$}
\put(8.1,2.5){$M$}
\put(1.35,-0.05){$\Q$}
\put(3.5,0.3){$G_K$}
\put(4.2,1.1){$K= K_{n_0},\  F=F_{n_0}$}
\put(4.2,0.6){$K' := K \cap \Q_{n_0}(\mu_q) \ \&\ 
\Q_{n_0}(\mu_q) = \Q(\mu^{}_{q p^{n_0}})$}
\put(4.2,0.1){(for $p=2$, $K' =  \Q_{n_0} \ \& \ 
\delta = 2$)}
\put(8.0,1.9){$\delta \vert \varphi(q)$}
\bezier{350}(1.75,0.05)(3.5,0.0)(4.0,1.2)
\bezier{350}(4.1,1.65)(4.35,2.05)(4.6,2.45)
\bezier{350}(7.6,1.65)(7.85,2.05)(8.1,2.45)
\bezier{350}(2.6,1.65)(2.85,2.05)(3.1,2.45)
\bezier{350}(4.2,5.7)(5.75,6.3)(7.3,5.7)
\put(5.65,6.15){${\mathcal T}_K$}
\bezier{500}(1.6,0.2)(2.5,4.0)(4.5,4.1)
\put(2.6,3.5){$G_n$}
\bezier{500}(4.2,1.55)(6.2,2.4)(7.25,5.36)
\put(6.5,3.5){${\mathcal G}_K$}
\end{picture}   }} $$
\unitlength=1.0cm

\medskip
\begin{lemma}\label{conductor}
Let $f_K$ be the conductor of $K$.
Then the conductor $f_{L_n}$ of $L_n$ ($n \geq n_0$) is 
equal to {\rm l.c.m.}\,$(f_K, q p^n)$. Thus for $n$ large enough (explicit), 
$f_{L_n}=q p^n f'$, with $p\nmid f'$.
If $p \nmid f_K$, then $f_{L_n} = q p^n f_K$ for all $n \geq n_0+e$.
\end{lemma}

\begin{proof} A classical formula (see, e.g., \cite[Proposition II.4.1.1]{Gr1}).
\end{proof}

\begin{lemma} Let $p^e$, $e \geq 0$, be the exponent of ${\mathcal T}_K$.
Then, for all $n \geq n_0+e$, the restriction 
${\mathcal T}_K \too {\rm Gal}(F_n/ K_n)$ is an isomorphism 
of $G_K$-modules and ${\mathcal T}_K \simeq {\rm Gal}(M_n/ L_n)$.
\end{lemma}

\begin{proof} The abelian group 
${\mathcal G}_K := {\rm Gal}(H_K^{\rm pr}/K)$ is normal in 
${\rm Gal}(H_K^{\rm pr}/\Q)$, then $({\mathcal G}_K)^{p^{n-n_0}}$ is normal; 
but $({\mathcal G}_K)^{p^{n-n_0}}$ fixes $F_n$ which is Galois over~$\Q$.
In other words, $G_K$, as well as ${\rm Gal}(K_n/\Q)$ or ${\rm Gal}(K_\infty/\Q)$, 
operate by conjugation in the same way since ${\mathcal G}_K$
is abelian; if $F$ is clearly non-unique, then $F_{n_0+e}$ is canonical, 
being the fixed fiel of $\big({\mathcal G}_K\big)^{p^e}$.
Then ${\rm Gal}(M_n/ L_n) \simeq {\rm Gal}(F_n/ K_n)$ is trivialy
an somorphism of $G_K$-modules.
\end{proof}

The use of the extension $F$ is not strictly necessary but clarifies the 
reasoning which needs to work at any level $n \geq n_0+e$ to preserve
Galois structures.

\smallskip
The extension $M_n/L_n$ (of exponent $p^e$) is a Kummer extension 
for the ``exponent'' $qp^n$ since $L_n$ contains the group $\mu_{qp^n}^{}$
and since $n \geq n_0+e$.

\smallskip
Let $G_n := {\rm Gal}(L_n/\Q)$ and let, for $n \geq n_0+e$,
$${\rm Rad}_n := \{w \in L_n^\times,\  \sqrt [{}^{qp^n} \!\!\!]{w} \in M_n\}$$

be the radical of $M_n/L_n$. Then we have the group isomorphism:
$${\rm Rad}_n/L_n^{\times qp^n} \simeq {\rm Gal}(M_n/L_n). $$
In some sense, the group ${\rm Rad}_n/L_n^{\times qp^n}$ 
does not depend on $n \geq n_0+e$ since the
canonical isomorphism
${\rm Gal}(M_{n+h}/L_{n+h}) \simeq {\rm Gal}(M_n/L_n)$
gives $L_{n+h}( \sqrt [{}^{qp^n} \!\!\!]{{\rm Rad}_n}) =M_{n+h}$;
the map ${\rm Rad}_n/L_n^{\times qp^n}  \ds\mathop{\too}^{p^h}
 {\rm Rad}_{n+h}/L_{n+h}^{\times q p^{n+h}}$
is an isomorphism for any $h\geq 0$. In other words, 
as soon as $n\geq n_0+e$, we have:
$${\rm Rad}_n \subseteq L_n^{\times qp^{n-e}}\ \ \& \ \ 
{\rm Rad}_{n+h} = {\rm Rad}_n^{p^h} \cdot L_{n+h}^{\times qp^{n+h}}. $$

\subsection{Spiegel involution}
The structures of $(\Z/q p^n \Z) [G_n]$-modules of 
the Galois group ${\rm Gal}(M_n/L_n)$ and 
${\rm Rad}_n/L_n^{\times qp^n}$ are related via the 
``Spiegel involution'' defined as follows:
let $\omega_n : G_n \too \Z/q p^n\Z$ be the {\it character of 
Teichm\"uller of level $n$} defined by:
$$\zeta^s = \zeta^{\omega_n(s)}, \ \hbox{for all $s \in G_n$ 
and all $\zeta \in \mu_{q p^n}^{}$.}$$

The Spiegel involution is the involution of $(\Z/q p^n \Z) [G_n]$ defined by:
$$\hbox{$x :=\sm_{s \in G_n} a_s \cdot s\ $ $\mapsto$
$\ x^* := \sm_{s \in G_n} a_s\cdot \omega_n(s)\cdot s^{-1}$.} $$

Thus, if $s$ is the Artin symbol $\Big(\Frac{L_n}{a} \Big)$, then 
$\Big(\Frac{L_n}{a} \Big)^{\! *} \equiv a\cdot \Big(\Frac{L_n}{a} \Big)^{-1}\!\!
\pmod {qp^n}$.
For the convenience of the reader we prove once again the very classical:

\begin{lemma}\label{duality}
Let $n \geq n_0 +e$ where $p^{n_0} = [K \cap \Q_\infty : \Q]$
and $p^e$ is the exponent of ${\mathcal T}_K$.
The annihilators $A_n$ of ${\rm Gal}(M_n/L_n)$ (thus of ${\mathcal T}_K$)
in $(\Z/q p^n \Z) [G_n]$ are the images of the annihilators $S_n$ of 
${\rm Rad}_n/L_n^{\times qp^n}$ by the Spiegel involution and inversely.
An annihilator $A_n$ of ${\mathcal T}_K$ only depends on its projection 
$A_{K,n}$ in $(\Z/q p^n \Z) [G_K]$.
\end{lemma}

\begin{proof} 
To simplify, put $\ov{\rm Rad} := {\rm Rad}_n/L_n^{\times qp^n}$\!,
${\mathcal T} := {\rm Gal}(M_n/L_n) \simeq {\mathcal T}_K$. Let: 
\begin{align*}
\lambda\  :\ \  & \ov{\rm Rad} \times {\mathcal T} \tooo \mu_{qp^n}^{} \\ 
               &\ \ \  (\ov w, \tau) \ \   \longmapsto   \big( \sqrt [{}^{qp^n} \!\!\!]{w}\big)^{\tau-1};
\end{align*}
then $\lambda$ is a non-degenerated  $\Z/q p^n \Z$-bilinear form such that:
$$\lambda(\ov w^s, \tau) = \lambda(\ov w, \tau^{s^*}), 
\ \, \hbox{for all $s \in G_n$, }$$ 
where $s^* = \omega_n(s)\cdot s^{-1}$
(see e.g., \cite[Corollary I.6.2.1]{Gr1}).

\smallskip
Let $S_n = \sm_{s \in G_n} a_s \cdot s \in (\Z/q p^n \Z) [G_n]$; then, for all
$(\ov w, \tau) \in \ov{\rm Rad} \times {\mathcal T}$ we have:
$$\lambda(\ov w^{S_n}, \tau) = \prd_{s \in G_n}\lambda(\ov w^s, \tau)^{a_s}
=  \prd_{s \in G_n}\lambda(\ov w, \tau^{s^*} )^{a_s} =
\lambda(\ov w, \tau^{S_n^*}). $$ 

So, if $S_n$ annihilates $\ov{\rm Rad}$, then $\lambda(\ov w, \tau^{S_n^*})=1$
for all $\ov w\ \&\ \tau$; since $\lambda$ is non-degenerated, 
$\tau^{S_n^*}=1$ for all $\tau \in {\mathcal T}$. Whence
the annihilation of ${\mathcal T}$ by $A_n=S_n^*$ {\it (without any
assumption on $K$ nor on $p$)}, then by 
the projection $A_{K,n}$ since ${\rm Gal}(L_n/K)$ acts trivially 
on ${\rm Gal}(M_n/L_n)$.
\end{proof}

\begin{remarks}\label{plus}
(i) As we have mention, the radical ${\rm Rad}_n$ does not depend
realy on the field $L_n$ for $n \geq n_0+e$; so, if we consider the
radical of the maximal
$p$-ramified abelian $p$-extension of $L_n$, of exponent $qp^n$:
$${\rm Rad}'_n := \{w' \in L_n^\times,\  L_n(\sqrt [{}^{qp^n} \!\!\!]{w'})/L_n \ 
\hbox{is $p$-ramified} \}, $$

we obtain a group whose $p$-rank tends to infinity with $n$;
this is due mainely to the $\Z_p$-rank of the compositum of the 
$\Z_p$-extensions of $L_n$ (totally imaginary) and
from the less known ${\mathcal T}_{L_n}$ which contains 
${\mathcal T}_{K_n}$.
But since ${\mathcal T}_K$ is annihilated by $1-s_\infty$, 
${\rm Rad}_n/L_n^{\times qp^n}$ is annihilated by 
$(1-s_\infty)^*=1+s_\infty$ which means that only the ``minus part'' 
of ${\rm Rad}'_n/L_n^{\times qp^n}$ is needed, which eliminates the 
huge ``plus'' part containing in particular all the units.
Thus ${\rm Rad}_n$ is essentially given by the ``relative'' 
$S'_n$-units of $L_n$ ($S'_n$ being the set of $p$-places of $L_n$) 
and generators of some ``relative'' $p$-classes of $L_n$.

\smallskip
(ii) In the case $p=2$, let ${\mathcal T}^{\rm res}_K :=
{\rm tor}_{\Z_2} \big({\mathcal G}_{K,S}^{\rm res\,ab}\big)$, where
${\mathcal G}^{\rm res}_{K,S}$ is the Galois group of the maximal abelian
$S$-ramified (i.e., unramified outside $2$ but possibly complexified) 
pro-$2$-extension of $K$ and let 
${\rm Rad}^{\rm res}_n$ the corresponding radical 
$\{w \in L_n^\times,\  \sqrt [{}^{4 \cdot 2^n} \!\!\!]{w} \in M^{\rm res}_n\}$, 
where $M^{\rm res}_n$ is analogous to $M_n$ for the restricted sense.
We observe that in the restricted sense, we have the exact sequence
\cite[Theorem III.4.1.5]{Gr1}
$0 \to (\Z/2\Z)^d \too {\mathcal T}^{\rm res}_K \too {\mathcal T}_K \to 1$,
then a dual exact sequence with radicals. As in \cite{All2}, one may consider 
more general ray class fields and find results of annihilation with suitable
Stickelberger or Solomon elements.
\end{remarks}

\section{Stickelberger elements and cyclotomic numbers}

\subsection{General definitions}
Let $f \geq 1$ be any modulus and let $\Q^f$ 
be the corresponding cyclotomic field $\Q(\mu_f)$.\,\footnote{Such modulus are
conductors of the corresponding cyclotomic fields, except for an even integer 
not divisible by $4$; but this point of view is essential to establish the 
functional properties of Stickelberger elements and cyclotomic numbers.
So, if $f$ is odd, we distinguish, by abuse, the notations 
$\Q^f$ and $\Q^{2f}$ despite their equality.}
Let $L$ be a subfield of $\Q^f$.

\smallskip
(i) We define (where all Artin symbols are taken over $\Q$):\par
\centerline{${\mathcal S}_{\Q^f} :=-\sm_{a=1}^{f} 
\Big(\Frac{a}{f}-\Frac{1}{2} \Big) \cdot \Big(\Frac{\Q^f}{a} \Big)^{-1}$, }

\noindent
and the restriction:
$${\mathcal S}_L := {\rm N}_{\Q^f/L} ({\mathcal S}_{\Q^f}) := 
-\sm_{a=1}^{f} \Big(\Frac{a}{f}-\Frac{1}{2} \Big) \cdot \Big(\Frac{L}{a} \Big)^{-1}$$
to $L$ of ${\mathcal S}_{\Q^f}$, 
where $a$ runs trough the integers $a \in [1, f]$ prime to $f$.
In this case, one must precise the relation between $f$ and the conductor
$f_L$ of $L$.

\smallskip
We know that the properties of annihilation of ideal classes need to multiply 
${\mathcal S}_L$ by an element of the ideal annihilator of
the group $\mu_f$ (or $\mu_{2 f}$), which is generated by $f$ (or $2\,f$)
and the multiplicators:
$$\delta_c := 1 - c\cdot \Big(\Frac{\,\Q^f}{c} \Big)^{-1}, $$
for $c$ odd, prime to $f$. This shall give integral elements in 
the group algebra.

\smallskip
(ii) Then we define in the same way:
$$\eta^{}_{\Q^f} := 1- \zeta_f  \, \  \&  \,  \  \eta^{}_{L} :=
{\rm N}_{\Q^f/L} (1- \zeta_f), \ f\ne 1, $$
where $\zeta_f$ is a primitive $f$th root of 
unity for which we assume the coherent definitions $\zeta_{f}^{m'}=\zeta_m$ 
if $f = m' \cdot m$.

\smallskip
It is well known that if $f$ is not a prime power, then
$\eta^{}_f$ is a unit, otherwise, ${\rm N}_{\Q^f/\Q} (1- \zeta_f) = \ell$
if $f=\ell^r$, $\ell \geq 2$ prime, $r\geq 1$.

\begin{definition} 
Since $\Frac{f-a}{f} - \Frac{1}{2} = - \big(\Frac{a}{f} - \Frac{1}{2}\big)$,
${\mathcal S}_{\Q^f} = {\mathcal S}'_{\Q^f}\cdot(1- s_\infty)$
and ${\mathcal S}_L ={\mathcal S}'_L \cdot (1- s_\infty)$,
where $s_\infty := \big(\frac{\Q^f}{-1} \big)$ is the complex conjugation, and~where:
$${\mathcal S}'_{\Q^f} :=  -\sm_{a=1}^{f/2} \Big(\Frac{a}{f}-\Frac{1}{2} \Big)
\cdot  \Big(\Frac{\Q^f}{a} \Big)^{-1}\ \ \&\ \ \
{\mathcal S}'_L :=  -\sm_{a=1}^{f/2} \Big(\Frac{a}{f}-\Frac{1}{2} \Big)
\cdot  \Big(\Frac{L}{a} \Big)^{-1}.$$ 
\end{definition}

\subsection{Norms of Stickelberger elements and cyclotomic numbers}\label{norms}
Let $f \geq 1$ and $m \mid f$ be any modulus and let $\Q^f$
and $\Q^m \subseteq \Q^f$ be the corresponding cyclotomic fields. 
Let ${\rm N}_{\Q^f/\Q^m}$ be the restriction map:
$$\Q[{\rm Gal}(\Q^f / \Q)] \too \Q[{\rm Gal}(\Q^m/\Q)], $$ 
or the usual arithmetic norm in $\Q^f/\Q^m$. Consider as above:
$${\mathcal S}_{\Q^f} := - \sm_{a=1}^{f}
\Big(\Frac{a}{f}-\Frac{1}{2} \Big) \cdot \Big(\Frac{\Q^f}{a} \Big)^{-1}
\ \ \& \ \  \eta^{}_{\Q^f} := 1- \zeta_f  \ (f \ne 1).$$
We have, respectively:
\begin{align}
{\rm N}_{\Q^f/\Q^m} ({\mathcal S}_{\Q^f}) & = 
\prd_{\ell \mid f,\  \ell \nmid m}
\Big(1-\Big(\Frac{\Q^m}{\ell} \Big)^{-1}\Big) 
\cdot {\mathcal S}_{\Q^m},  \\
{\rm N}_{\Q^f/\Q^m} (\eta^{}_{\Q^f}) &= 
\big( \eta^{}_{\Q^m} \big)_{}^{\prod_{\ell \mid f,\  \ell \nmid m}
\big(1-\big(\frac{\Q^m}{\ell} \big)^{-1}\big)} \ \,\hbox{if $m\ne 1$}.
\end{align}\label{uc}

As we have explained in the previous footnote, if $m$ is odd, 
then we have:
$$\hbox{${\rm N}_{\Q^{2m}/\Q^m} ({\mathcal S}_{\Q^{2m}})  = 
\big(1-\big(\frac{\Q^m}{2} \big)^{-1}\big) 
\cdot {\mathcal S}_{\Q^m}$,  \ \ \ 
${\rm N}_{\Q^{2m}/\Q^m} (\eta^{}_{\Q^{2m}}) = 
\eta_{\Q^m}^{\big(1-\big(\frac{\Q^m}{2} \big)^{-1}\big)}$,}$$
where the ``norms'' ${\rm N}_{\Q^{2m}/\Q^m}$ are of course
the identity. For instance one verifies immediately that 
${\mathcal S}_{\Q^6}=\frac{1}{3} (1-s_\infty)$ and
${\mathcal S}_{\Q^3}=\frac{1}{6} (1-s_\infty)$, but since $2$ is inert
in $\Q^3/\Q$, $\big(1-\big(\frac{\Q^3}{2} \big)^{-1}\big) = 1-s_\infty$
and one must compute $(1-s_\infty) {\mathcal S}_{\Q^3} 
=\frac{1}{6} (1-s_\infty)^2=\frac{1}{3} (1-s_\infty)$ as expected.
We have ${\mathcal S}_{\Q^2}=0$ and 
${\mathcal S}_{\Q^1} =-\frac{1}{2}$.

\medskip
If $L$ (imaginary or real), of conductor $f$,
is an extension of $k$, of conductor $m \mid f$, let
${\mathcal S}_L := {\rm N}_{\Q^f/L}({\mathcal S}_{\Q^f})$ and
$\eta^{}_L :=  {\rm N}_{\Q^f/L}( \eta^{}_{\Q^f})$, then:
\begin{align*}\label{st2}
{\rm N}_{L/k}({\mathcal S}_{L}) & = 
\prd_{\ell \mid f,\  \ell \nmid m}
\Big(1-\Big(\Frac{k}{\ell} \Big)^{-1}\Big) 
\cdot {\mathcal S}_{k}, \\
{\rm N}_{L/k}({\mathcal S}'_{L}) & \equiv 
\prd_{\ell \mid f,\  \ell \nmid m}
\Big(1-\Big(\Frac{k}{\ell} \Big)^{-1}\Big) 
\cdot {\mathcal S}'_{k} \pmod{(1+s_\infty) \cdot \Q[G_k]}, \\
{\rm N}_{L/k}( \eta^{}_L) &= \ds
( \eta^{}_k)_{}^{\prod_{\ell \mid f,\  \ell \nmid m} 
\big(1-\big(\frac{k}{\ell} \big)^{-1} \big)} \ \,\hbox{if $m\ne 1$ 
(i.e., $k\ne \Q$)}.
\end{align*}

If $f=\ell^r$, $\ell$ prime, $r\geq 1$,
then ${\rm N}_{\Q^f/\Q}(\eta^{}_{\Q^f}) = \ell$,
otherwise ${\rm N}_{\Q^f/\Q}(\eta^{}_{\Q^f}) = 1$.

\medskip
This implies that ${\rm N}_{L/k}({\mathcal S}_L) = 0$ 
(resp. ${\rm N}_{L/k}( \eta^{}_L) = 1$) as soon as there exists 
a prime $\ell \mid f$, $\ell \nmid m$, totally split in $k$. 
In particular, if $k$ is real, the formula is valid for the infinite 
place and ${\rm N}_{L/k}({\mathcal S}_L) = 0$ (of course,
if $L \ne \Q$ is real, $S_L=0$).

\smallskip
For the classical proofs, we consider by induction the case $f=\ell \cdot m$,
with $\ell$ prime and examine the two cases $\ell \mid m$ and $\ell \nmid m$;
the case of Stickelberger 
elements been crucial for our purpose, we give again a proof
(a similar reasoning will be detailed for the Theorem \ref{thmfond}).

\smallskip
To simplify, put ${\mathcal S}_{\Q^f} =: {\mathcal S}_f$,
${\mathcal S}_{\Q^m} =: {\mathcal S}_m$, and consider:
$${\mathcal S}_f = -\sm_{a=1}^{f} 
\Big(\Frac{a}{f}-\Frac{1}{2} \Big) \cdot \Big(\Frac{\Q^f}{a} \Big)^{-1}, $$
for $f= \ell \cdot m$, $\ell \nmid m$,
where $a$ runs trough the integers $a \in [1, f]$ prime to $f$.

\smallskip
Put $a=b+ \lambda \cdot m$, $b\in [1, m]$, $\lambda \in [0, \ell-1]$; 
since $a$ must be prime to $f$, $b$ is automatically prime to $m$ but
we must exclude $\lambda_b^* \in [0, \ell-1]$ such that:
$$\hbox{$b + \lambda_b^*\cdot m = b'_\ell \cdot  \ell, \ \, b'_\ell \in [1, m]
 \  \hbox{($b'_\ell$ is necessarily prime to $m$)}$.}$$
We then have:
\begin{align*}
{\rm N}_{\Q^f /\Q^m} & ({\mathcal S}_f) \\
&= -\sm_{a=1}^{f} 
\Big(\Frac{a}{f}-\Frac{1}{2} \Big) \cdot \Big(\Frac{\Q^m}{a} \Big)^{-1} 
 =  -\sm_{b, \,\lambda \ne \lambda_b^*} 
\Big(\Frac{b + \lambda \, m}{\ell\, m}-\Frac{1}{2} \Big) \cdot \Big(\Frac{\Q^m}{b} \Big)^{-1} \\
& = -\sm_{b} \Big(\Frac{\Q^m}{b} \Big)^{-1}\sm_{\lambda \ne \lambda_b^*} 
  \Big(\Frac{b}{\ell\, m}  + \Frac{\lambda}{\ell}-\Frac{1}{2} \Big) \\
& =  -\sm_{b} \Big(\Frac{\Q^m}{b} \Big)^{-1} \Big(\Frac{\ell-1}{\ell} \Frac{b}{ m} 
- \Frac{\ell - 1}{2}  \Big) - \sm_{b} \Big(\Frac{\Q^m}{b} \Big)^{-1} \Frac{1}{\ell}
\Big(\Frac{\ell \,(\ell - 1)}{2} - \lambda_b^* \Big)\\
& = -  \Big(1 - \Frac{1}{\ell} \Big) \sm_{b} \Big(\Frac{\Q^m}{b} \Big)^{-1} \Frac{b}{m} +
\Frac{1}{\ell}  \sm_{b} \Big(\Frac{\Q^m}{b} \Big)^{-1} \lambda_b^* \,.
\end{align*}

Since the correspondence $b \mapsto b'_\ell$ is bijective on the set of 
elements prime to $m$ in $[1, m]$, one has, with 
$\lambda_b^* = \Frac{b'_\ell \cdot \ell -b}{m}$ and  
$\Big(\Frac{\Q^m}{b} \Big) = \Big(\Frac{\Q^m}{b'_\ell} \Big)
\Big(\Frac{\Q^m}{\ell} \Big)$: 
\begin{align*}
 \Frac{1}{\ell}  \sm_{b} \Big(\Frac{\Q^m}{b} \Big)^{-1} \lambda_b^*  
 &=  \sm_{b} \Big(\Frac{\Q^m}{b} \Big)^{-1} \Frac{b'_\ell}{m} -
  \Frac{1}{\ell} \sm_{b} \Big(\Frac{\Q^m}{b} \Big)^{-1} \Frac{b}{m} \\
& = \Big(\Frac{\Q^m}{\ell} \Big)^{-1} \sm_{b} \Big(\Frac{\Q^m}{b'_\ell} \Big)^{-1}
   \Frac{b'_\ell}{m} - \Frac{1}{\ell}  \sm_{b} \Big(\Frac{\Q^m}{b} \Big)^{-1}\Frac{b}{m} \\
& = \Big(  \Big(\Frac{\Q^m}{\ell} \Big)^{-1} - \Frac{1}{\ell} \Big) \cdot 
   \sm_{b} \Big(\Frac{\Q^m}{b} \Big)^{-1} \Frac{b}{m}\,.
\end{align*}

Thus we obtain:
\begin{align*}
{\rm N}_{\Q^f/\Q^m} ({\mathcal S}_f)  &=
- \Big(1 -  \Frac{1}{\ell} \Big)  \sm_{b} \Big(\Frac{\Q^m}{b} \Big)^{-1} \Frac{b}{m}+
 \Big(  \Big(\Frac{\Q^m}{\ell} \Big)^{-1} - \Frac{1}{\ell} \Big) \cdot 
   \sm_{b} \Big(\Frac{\Q^m}{b} \Big)^{-1} \Frac{b}{m} \\
& =  -  \Big(1 -  \Big(\Frac{\Q^m}{\ell} \Big)^{-1} \Big) \sm_{b} 
  \Big(\Frac{\Q^m}{b} \Big)^{-1}\Frac{b}{m}\,.
\end{align*}

But $\Frac{1}{2}  \sm_{b} \Big(\Frac{\Q^m}{b} \Big)^{-1}\!
\Big(1 -  \Big(\Frac{\Q^m}{\ell} \Big)^{-1} \Big) = 0$; so replacing
$\Frac{b}{m}$ by $\Frac{b}{m} - \Frac{1}{2}$ we get:
$${\rm N}_{\Q^f/\Q^m} ({\mathcal S}_f)= 
\Big(1 - \Big(\Frac{\Q^m}{\ell}\Big)^{-1}\Big)
\cdot {\mathcal S}_m. $$

Then it is easy to compute that if $\ell \mid m$, any $\lambda \in [0, \ell-1]$
is suitable, giving:
$${\rm N}_{\Q^f/\Q^m} ({\mathcal S}_f)={\mathcal S}_m. $$

The case of cyclotomic elements $\eta_f$ is exactly the same, 
replacing the additive setting by the multiplicative one.

\subsection{Multiplicators of Stickelberger elements}
The conductor of $L_n$, $n \geq n_0$, is $f_{L_n} = {\rm l.c.m.\,}(f_K, q p^n)$ (Lemma 
\ref{conductor}). So in general $f_{L_n} = q p^n\cdot f'$ with $p \nmid f'$,
except if $f_K$ is divisible by a large power of $p$ in which case one 
must take $n$ large enough in the practical computations (write
$f_K = qp^{n_0+r} f'$, $r \geq 0$, and take $n \geq n_0+r$).
In some formulas we shall abbreviate $f_{L_n}$ by $f_n$.

\smallskip
Let $c$ be an (odd) integer, prime to $f_n$, and let:
\begin{equation}\label{mult1}
{\mathcal S}_{L_n}(c) := \Big(1 - c \,\Big(\Frac{L_n}{c} \Big)^{-1} \Big)\cdot 
{\mathcal S}_{L_n}.
\end{equation}

Then ${\mathcal S}_{L_n}(c) \in \Z[G_n]$ as we have explain; indeed, we have:
$${\mathcal S}_{L_n} (c)=\Frac{-1}{f_n} \sm_a 
\Big[a\, \Big(\Frac{L_n}{a} \Big)^{-1} - a c \,\Big(\Frac{L_n}{a} \Big)^{-1}
\Big(\Frac{L_n}{c} \Big)^{-1}\Big]
+ \Frac{1-c}{2} \sm_a \Big(\Frac{L_n}{a} \Big)^{-1}.$$

Let $a'_{c} \in [1, f_n]$ be the unique integer such that 
$a'_{c} \cdot c \equiv a \pmod {f_n}$
and put $a'_{c} \cdot  c = a + \lambda^n_a(c) f_n$, 
$\lambda^n_a(c) \in \Z$; then, using the bijection 
$a \mapsto a'_{c}$ in the second summation and the fact that
$\Big(\Frac{L_n}{a'_{c}} \Big) \Big(\Frac{L_n}{c} \Big) = \Big(\Frac{L_n}{a} \Big)$,
this yields:
\begin{equation*}
\begin{aligned}
 {\mathcal S}_{L_n}(c)& 
=\Frac{-1}{f_n}  \Big[ \sm_a a \, \Big(\Frac{L_n}{a} \Big)^{-1}\!\! - 
\sm_a a'_{c} \cdot c \,\Big(\Frac{L_n}{a'_{c}} \Big)^{-1} \! \Big(\Frac{L_n}{c} \Big)^{-1}\Big] 
 + \Frac{1-c}{2} \sm_a \Big(\Frac{L_n}{a} \Big)^{-1} \\
&=\Frac{-1}{f_n} \sm_a \Big [a - a'_{c} \cdot  c \Big] \Big(\Frac{L_n}{a} \Big)^{-1}
+ \Frac{1-c}{2}\sm_a \Big(\Frac{L_n}{a} \Big)^{-1} \\
& =  \sm_a \Big[\lambda^n_a(c)+ \Frac{1-c}{2} \Big]
 \Big(\Frac{L_n}{a} \Big)^{-1} \in \Z[G_n].
\end{aligned}
\end{equation*}

\begin{lemma}\label{sprime} We have the relations
$\lambda^n_{f_n-a}(c) + \frac{1-c}{2} = -\big(\lambda^n_{a}(c) + \frac{1-c}{2} \big)$
for all $a \in [1, f_n]$ prime to $f_n$. Then:
\begin{equation}\label{mult2}
{\mathcal S}'_{L_n}(c) :=  \sm_{a=1}^{f_n/2} \Big[\lambda^n_a(c)+ \Frac{1-c}{2} \Big]
\Big(\Frac{L_n}{a} \Big)^{-1} \in \Z[G_n]
\end{equation} 
is such that
${\mathcal S}_{L_n}(c) = {\mathcal S}'_{L_n}(c) \cdot (1-s_\infty)$, whence
${\mathcal S}_{L_n}(c)^* = {\mathcal S}'_{L_n}(c){}^* \cdot (1+s_\infty)$.
\end{lemma}

\begin{proof} By definition, the integer $(f_n-a)'_c$ is in $[1, f_n]$ and 
congruent modulo $f_n$ to
$(f_n-a) \,c^{-1} \equiv - a c^{-1} \equiv - a'_c \pmod {f_n}$; thus
$(f_n-a)'_c =f_n-  a'_c$ and 
$$\lambda^n_{f_n-a}(c) =\Frac{(f_n-a)'_c \,c- (f_n-a)}{f_n}
= \Frac{(f_n-a'_c)\, c- (f_n-a)}{f_n} = c-1- \lambda^n_a(c), $$

whence $\lambda^n_{f_n-a}(c) + \frac{1-c}{2} =
 - \big( \lambda^n_{a}(c) + \frac{1-c}{2}  \big)$ and the result.
\end{proof}

The multiplicator $\delta_c := \big(1 - c \,\big(\frac{L_n}{c} \big)^{-1} \big)$ has 
a great importance since the image of $\delta_c $ by the Spiegel involution is 
$\delta_c^* := 1 - \big(\frac{L_n}{c} \big) \pmod{qp^n}$;
the order of the Artin symbol of $c$ shall be crucial.

\section{Annihilation of radicals and Galois groups}

\subsection{Annihilation of ${\rm Rad}_n/L_n^{\times qp^n}$}
We begin with the classical property of annihilation of class groups
of imaginary abelian fields by modified Stickelberger elements
${\mathcal S}_{L_n}(c) = \delta_c \cdot {\mathcal S}_{L_n}$.
Before let's give two technical lemmas.
Recall that ${\mathcal S}_{L_n}(c) = {\mathcal S}'_{L_n}(c)\cdot (1-s_\infty)$
and that, from \S\,\ref{norms}, the ${\mathcal S}_{L_n}$,
${\mathcal S}_{L_n}(c)$ and ${\mathcal S}'_{L_n}(c) 
\pmod {(1+s_\infty) \Z[G_n]}$ form coherent families in 
$\ds\varprojlim_{n \geq n_0+e} \Q[G_n]$ for the ``norm'' since
$f_{L_n}$ and $f_{L_{n+h}}$ are divisible by the same prime 
numbers for all $h \geq 0$.

\begin{lemma}\label{lemma1}
Let $\zeta \in \mu_{qp^n}$, $n \geq n_0+e$. 
If $\zeta \in {\rm Rad}_n$ (or ${\rm Rad}^{\rm res}_n$ when $p=2$) 
then $\zeta = 1$.
\end{lemma}

\begin{proof} If $\zeta\ne 1$ with $L_n(\sqrt [{}^{qp^n} \!\!\!]{\zeta}) 
\subseteq M_n$ (or $M^{\rm res}_n$), we would have 
$L_n(\sqrt [{}^{qp^n} \!\!\!]{\zeta})= L_{n+h}$, where $h\geq 1$
since $\sqrt [{}^{qp^n} \!\!\!]{\zeta}$ is of order $\geq qp^{n+1}$
and since $\mu_{p^\infty} \cap L_n^\times = \mu_{qp^n}$,
which is absurd because of the linear disjonction 
$L_{n+h} \cap M_n=L_n$ (or $L_{n+h} \cap M^{\rm res}_n=L_n$).
\end{proof}

\begin{lemma} \label{lemma2}
Let $w_0 \in {\rm Rad}_n$ be real. Then $w_0^2 \in L_n^{\times qp^n}$.
\end{lemma}

\begin{proof} Since $K$ is real, we know that
$1-s_\infty$ annihilates the $(\Z/ qp^n \Z) [G_n]$-module
${\rm Gal}(M_n/L_n)$, thus $1+s_\infty$ annihilates 
${\rm Rad}_n/L_n^{\times  qp^n}$ and $w_0^{1+s_\infty}=w_0^2
\in  L_n^{\times  qp^n}$ (this does not work for the restricted sense
since the minus part of ${\mathcal T}^{\rm res}_K$ is of order $2^d$).
\end{proof}

\begin{theorem} Let $p^e$ be the exponent of ${\mathcal T}_K :=
{\rm tor}_{\Z_p} \big({\mathcal G}_{K,S}^{\rm ab}\big)$ 
($p$-ramification in the ordinary sense). For $p=2$, let $2^{e^{\rm res}}$ 
be the exponent of ${\mathcal T}^{\rm res}_K :=
{\rm tor}_{\Z_2} \big({\mathcal G}_{K,S}^{\rm res\,ab}\big)$, where
${\mathcal G}^{\rm res}_{K,S}$ is the Galois group of the maximal $S$-ramified 
in the restricted sense (i.e., unramified outside $2$ but complexified) 
pro-$2$-extension of $K$ and let ${\rm Rad}^{\rm res}_n$ be 
the corresponding radical.

\smallskip
(i) $p>2$. For all $n \geq  n_0+e$, the $(\Z/qp^n \Z)[G_n]$-module 
${\rm Rad}_n/L_n^{\times qp^n}$ is annihilated by 
${\mathcal S}'_{L_n}(c)$. Thus, 
${\mathcal S}'_{L_n}(c){}^*$ annihilates ${\mathcal T}_K$.

\smallskip
(ii) $p=2$, ordinary sense. The annihilation occurs with $2\,{\mathcal S}_{L_n}(c)$ 
and with $4\,{\mathcal S}'_{L_n}(c)$. Thus $2\,{\mathcal S}_{L_n}(c){}^*$ and 
$4\,{\mathcal S}'_{L_n}(c){}^*$ annihilate ${\mathcal T}_K$.

\smallskip
(iii) $p=2$, restricted sense. For all $n \geq  n_0+e^{\rm res}$, the 
$(\Z/4 \cdot 2^n \Z)[G_n]$-module 
${\rm Rad}^{\rm res}_n/L_n^{\times 4 \cdot 2^n}$ is annihilated by 
$2\,{\mathcal S}_{L_n}(c)$; thus $2\,{\mathcal S}_{L_n}(c){}^*$
annihilates ${\mathcal T}^{\rm res}_K$.
\end{theorem}

\begin{proof} Let $w \in {\rm Rad}_n$; since
$L_n(\sqrt [{}^{qp^n} \!\!\!]{w})/L_n$ is $p$-ramified, 
$(w) = {\mathfrak a}^{qp^n}\! \cdot {\mathfrak b}$
where ${\mathfrak a}$ is an ideal of $L_n$, prime to $p$, and ${\mathfrak b}$
is a product of prime ideals ${\mathfrak p}_n$ of $L_n$ dividing~$p$. 
Let ${\mathfrak p}_n \mid {\mathfrak b}$
and consider ${\mathfrak p}_n^{{\mathcal S}_{L_n}(c)}$; one can replace 
${\mathcal S}_{L_n}(c)$ by its restriction to the decomposition field $k$ 
(possibly $k=\Q$) of $p$ in the abelian extension $L_n/\Q$, which gives rise to
the Euler factor $1-\Big(\Frac{k}{p} \Big)^{-1}$ since $k$, of
conductor prime to $p$, is strictely contained in $L_n$
of conductor $qp^n f'$ for $n \geq  n_0+e$; 
so this factor is $0$ and ${\mathfrak b}^{{\mathcal S}_{L_n}(c)}=1$.

\smallskip
From the principality of the ideal ${\mathfrak a}^{{\mathcal S}_{L_n}(c)}$
(Stickelberger's theorem) there exists
$\alpha_n \in L_n^\times$ and a unit $\varepsilon_n$ of $L_n$ such that:
\begin{equation}\label{rad}
w^{{\mathcal S}_{L_n}(c)} = \alpha_n^{qp^n} \!\!\cdot \varepsilon_n. 
\end{equation}

We see that $\varepsilon_n^{1+s_\infty}$ is the $q p^n$th power of 
a unit of $L_n$: consider $\varepsilon_n^{1+s_\infty}$ in \eqref{rad}
with the fact that ${\mathcal S}_{L_n}(c)={\mathcal S}'_{L_n}(c) (1-s_\infty)$.
Since the $\Z$-rank of the groups of units of $L_n$ and $L_n^+$ (the 
maximal real subfield of $L_n$) are equal, a power $\varepsilon_n^N$ of 
$\varepsilon_n$ is a real unit; so 
$\varepsilon_n^{1-s_\infty}$ is a torsion element and
$\varepsilon_n^2 = \varepsilon_n^{1+s_\infty}
\varepsilon_n^{1-s_\infty}$ is equal, up to a $qp^n$th power, to a $p$-torsion 
element of the form $\zeta' \in {\rm Rad}_n$. Thus $\zeta'=1$ (Lemma \ref{lemma1})
and $\varepsilon_n^2 \in L_n^{\times qp^n}$.
 
\smallskip
(i) Case $p\ne 2$. We deduce from the above
that $\varepsilon_n \in L_n^{\times p^{n+1}}$.
We have $w^{{\mathcal S}'_{L_n}(c)(1-s_\infty)} = \beta_n^{p^{n+1}}$; 
but $\beta_n^{1+s_\infty} = 1$ (the property is also true for $p=2$ since 
the result is a totally positive root of unity in $L_n^+$, but the 
proof only works taking the square of the relation \eqref{rad} using
$\varepsilon_n^2$), and there exists $\gamma_n \in L_n^\times$ 
such that $\beta_n=\gamma_n^{1-s_\infty}$, and
$w^{{\mathcal S}'_{L_n}(c)} \cdot \gamma_n^{-p^{n+1}} = w_0$,
a real number in the radical, thus a $p^{n+1}$th power 
(Lemma \ref{lemma2}) (as above, the proof for $p=2$ only works 
taking once again the square of this relation to get $w_0^2$).
Other proof for any $p\geq 2$: since ${\mathcal T}_K$ is annihilated by $1-s_\infty$,
${\rm Rad}_n/L_n^{\times qp^n}$ is annihilated by $1+s_\infty$, thus
$w^{1-s_\infty} \in w^2 \cdot L_n^{\times qp^n}$ for all $w \in {\rm Rad}_n$,
and $w^{{\mathcal S}_{L_n}(c)} = w^{2\, {\mathcal S}'_{L_n}(c)}$ up to 
$L_n^{\times qp^n}$.

\medskip
(ii) Case $p=2$ in the ordinary sense (so $L_n^+ = K_n$). 
The result is obvious taking the square
in the previous computations giving $\varepsilon_n^2$ instead of $\varepsilon_n$
for the annihilation with $2\,{\mathcal S}_{L_n}(c)$, then 
$w_0^2$ for the annihilation with $4\,{\mathcal S}'_{L_n}(c)$.

\medskip
(iii) Case $p=2$ in the restricted sense. The proof is in fact contained in the 
same relation $(w) = {\mathfrak a}^{4 \cdot 2^n}\! \cdot {\mathfrak b}$,
for all $w \in {\rm Rad}^{\rm res}_n$,
where ${\mathfrak a}$ is an ideal of $L_n$, prime to $2$, and ${\mathfrak b}$
is a product of prime ideals ${\mathfrak p}_n$ of $L_n$ dividing $2$, then the 
relation \eqref{rad}, $n\geq n_0+e$. 
\end{proof}

\subsection{Computation of ${\mathcal S}_{L_n}(c)^*$ or 
${\mathcal S}'_{L_n}(c)^*$ -- Annihilation of ${\mathcal T}_K$} \label{slnc}

From the expressions \eqref{mult1} and \eqref{mult2} of ${\mathcal S}_{L_n}(c)$, 
the image by the Spiegel involution is:
\begin{equation*}
{\mathcal S}_{L_n}(c)^* \equiv  
 \sm_{a=1}^{f_n}  \Big[ \lambda^n_a(c) + \Frac{1-c}{2} \Big]\, 
 a^{-1} \Big(\Frac{L_n}{a} \Big) \pmod {q p^n},
\end{equation*}
which defines a coherent familly  $({\mathcal S}_{L_n}(c)^*)_n 
\in \ds\varprojlim_{n \geq n_0+e} \Z/qp^n \Z[G_n]$ of annihilators of the
Galois groups ${\rm Gal}(M_n/L_n) \simeq {\mathcal T}_K$.
In the case $p\ne 2$, one may use equivalently ${\mathcal S}'_{L_n}(c){}^*$
with the half summation.

\smallskip
Since the operation of ${\rm Gal}(L_n/K)$ on ${\rm Gal}(M_n/L_n)$ is trivial, 
by restriction of ${\mathcal S}_{L_n}(c)^*$ to $K$ (see Lemma \ref{duality}),
one obtains a coherent familly of annihilators of ${\mathcal T}_K$ denoted 
$({\mathcal A}_{K,n}(c))_n \in \ds\varprojlim_{n \geq n_0+e} \Z/qp^n \Z[G_K]$,
whose $p$-adic limit:

\medskip
\centerline{${\mathcal A}_K(c) := \ds\lim_{n \to \infty}{\mathcal A}_{K,n}(c) 
= \ds\lim_{n \to \infty} \sm_{a=1}^{f_n} \Big[ \lambda^n_a(c) + \Frac{1-c}{2} \Big]\, 
 a^{-1} \Big(\Frac{K}{a} \Big) \in \Z_p[G_K] $}
 
\medskip
is a canonical annihilator of ${\mathcal T}_K$ that we shall link to
$p$-adic $L$-functions; of course, it is sufficient to know its coefficients 
modulo the exponent $p^e$ of ${\mathcal T}_K$ and in a programming 
point of view, the element ${\mathcal A}_{K,n_0+e}(c)$ annihilates
${\mathcal T}_K$, knowing that \cite[Program I, \S\,3.2]{Gr3} gives the
group structure of ${\mathcal T}_K$.

\begin{remark}
Let $\alpha_{L_n}:= \sm_{a=1}^{f_n} a^{-1}  \Big(\Frac{L_n}{a} \Big) \equiv
\Big[\sm_{a=1}^{f_n}  \Big(\Frac{L_n}{a} \Big)^{-1} \Big]^*$; we have:

\centerline{$\alpha_{L_n} := \!\sm_{a=1}^{f_n/2} a^{-1}  \Big(\Frac{L_n}{a} \Big) +
(f_n-a)^{-1}  \Big(\Frac{L_n}{f_n-a} \Big) \! \equiv\! 
\sm_{a=1}^{f_n/2} a^{-1} \Big(\Frac{L_n}{a} \Big)(1-s_\infty) \!\!\!\pmod {f_n}$}

\smallskip
which annihilates ${\mathcal T}_K$ and is such that 
${\rm N}_{L_n/K}(\alpha_{L_n}) \equiv 0 \pmod {q p^n}$
since $K$ is real.
We shall neglect such expressions and use the symbol
$\ \wt{\equiv}\ $, where $A \ \wt\equiv\  B \pmod {p^{n+1}}$ 
will mean $A = B + \mu \cdot p^{n+1}
+ \nu \cdot \sm_{a=1}^{f_n} a^{-1} \Big(\Frac{K}{a} \Big)$,
in the group algebra $\Z_p[G_K]$, $\mu, \nu$ in $\Z_p$
(we put the modulus $p^{n+1}$ instead of $qp^n$ to cover, 
subsequently, the case $p=2$; moreover, $p^{n+1}$ annihilates 
${\mathcal T}_K$ since $n \geq n_0+e$). By abuse, we still denote 
${\mathcal A}_K(c) = \ds\lim_{n \to \infty}\hbox{$\sum_{a=1}^{f_n} 
 \lambda^n_a(c) \, a^{-1} \big(\frac{K}{a} \big)$}$.
\end{remark}

Thus, we have obtained:

\begin{theorem}\label{st*}
Let $c$ be any integer prime to $2p$ and to the conductor 
of $K$. 

Assume $n \geq n_0+e$ and let $f_n$
be the conductor of $L_n$; for
all $a \in [1, f_n]$, prime to $f_n$, let $a'_{c}$ be the 
unique integer in $[1, f_n]$ such that $a'_{c} \cdot c \equiv a \pmod {f_n}$ and put
$a'_{c} \cdot  c - a = \lambda^n_a(c)\, f_n$, $\lambda^n_a(c) \in \Z$.

Let ${\mathcal A}_{K,n}(c) := \!\sm_{a=1}^{f_n} \lambda^n_a(c) \, 
a^{-1} \Big(\Frac{K}{a} \Big)$
and put ${\mathcal A}_{K,n}(c) = {\mathcal A}'_{K,n}(c) \cdot (1+s_\infty)$
where ${\mathcal A}'_{K,n}(c) = \!\sm_{a=1}^{f_n/2} \lambda^n_a(c) \, 
a^{-1} \Big(\Frac{K}{a} \Big)$.
Let ${\mathcal A}_K(c) :=  \ds\lim_{n \to \infty}\ \Big[
\hbox{$\!\sm_{a=1}^{f_n}$} \lambda^n_a(c) \, a^{-1} \Big(\Frac{K}{a} \Big)\Big]$
and put ${\mathcal A}_K(c) =: {\mathcal A}'_K(c) \cdot (1+s_\infty)$.

\medskip
(i) For $p\ne 2$, ${\mathcal A}'_K(c)$ annihilates the $\Z_p[G_K]$-module
${\mathcal T}_K$. 

\medskip
(ii) For $p=2$, the annihilation is true for $2 \cdot {\mathcal A}_K(c)$
and $4 \cdot {\mathcal A}'_K(c)$.
\end{theorem}

In practice, when the exponent $p^e$ is known, one can use 
$n=n_0+e$ and the annihilators ${\mathcal A}_{K,n}(c)$ or 
${\mathcal A}'_{K,n}(c)$, the annihilator limit ${\mathcal A}_K(c)$ 
being related to $p$-adic $L$-functions
of primitive characters, thus giving the other approach
than Solomon one, that we shall obtain in Theorem \ref{thmp}. 

\begin{remark}  We have proved in a seminar report (1977) that for 
$p=2$, ${\mathcal S}'_{L_n}(c)$ annihilates $\Cl_{L_n}/\Cl_{L_n}^{\,0}$, 
where $\Cl_{L_n}$ is the $2$-class group of $L_n$ and 
where $\Cl_{L_n}^{\,0}$ is generated by the classes
of the {\it the invariant ideals} in $L_n/K_n$. 

\smallskip
This shows that some $2$-classes may give an 
obstruction; but ${\rm Rad}_n$ is particular as we have 
explained in Remark \ref{plus}. 
In \cite{Gre}, Greither gives suitable statements
about Stickelberger's theorem for $p=2$, using the main theorems 
of Iwasawa's theory about the orders $\frac{1}{2} L_2(1,\chi)$ 
of the isotypic components.
\end{remark}

From this, as well as some numerical experiments, and the 
roles of $\varepsilon_n$ and $w_0$ in the above reasonings,
we may propose the following conjecture:

\begin{conjecture} \label{conj}
Let $p=2$ and let $K$ be a real abelian number field linearly 
disjoint from the cyclotomic $\Z_2$-extension. 
Put ${\mathcal A}_K(c) = {\mathcal A}'_K(c) \cdot (1+ s_\infty)$
(see formula of Theorem \ref{st*}).
Then  ${\mathcal A}'_K(c)$ annihilates ${\mathcal T}_K$.
\end{conjecture}

If there exists, in the class of ${\mathcal A}'_K(c)$ modulo 
$\sum_{\sigma \in G_K} \sigma$, an element of the form 
$2 \cdot {\mathcal A}''_K(c)$, ${\mathcal A}''_K(c) \in \Z_p[G_K]$, 
one may ask if ${\mathcal A}''_K(c)$ does annihilate ${\mathcal T}_K$.
We shall give a counterexample for the annihilation of ${\mathcal T}_K$ 
by ${\mathcal A}''_K(c)$ (see \S\,\ref{555}), but we ignore if this may be
true under some assumptions.

\subsection{Experiments for cyclic cubic fields with 
$p \equiv 1 \pmod 3$}\label{cubiccase}
To simplify we suppose $f_K$ prime. The first part of the program gives 
a defining polynomial.
A second part computes the $p$-adic valuation of $\order{\mathcal T}_K$ 
using \cite[Program I, \S\,3.2]{Gr3} and gives ${\mathcal A}_K(c) 
= \Lambda_0 + \Lambda_1 \sigma^{-1} + \Lambda_2 \sigma^{-2}$ 
modulo a power of $p$, after the choice of $c$, 
prime to $2 p f_K$, with an Artin symbol of order $3$; in the program
${\sf p^{ex}}$ is the exponent $p^e$ of ${\mathcal T}_K$
and ${\sf fn}$ the conductor of $L_n$. The parameter ${\sf nt}$
must be $>{\sf ex}$.

\smallskip\footnotesize
\begin{verbatim}
{p=7;nt=8;forprime(f=7,10^4,if(Mod(f,3)!=1,next);
for(bb=1,sqrt(4*f/27),if(vf==2 & Mod(bb,3)==0,next);A=4*f-27*bb^2;
if(issquare(A,&aa)==1,if(Mod(aa,3)==1,aa=-aa);
P=x^3+x^2+(1-f)/3*x+(f*(aa-3)+1)/27;K=bnfinit(P,1);Kpn=bnrinit(K,p^nt);
C5=component(Kpn,5);Hpn0=component(C5,1);Hpn=component(C5,2);
h=component(component(component(K,8),1),2);L=List;ex=0;
i=component(matsize(Hpn),2);R=0;for(k=1,i-1,co=component(Hpn,i-k+1);
if(Mod(co,p)==0,R=R+1;val=valuation(co,p);if(val>ex,ex=val);
listinsert(L,p^val,1)));Hpn1=component(Hpn,1);
vptor=valuation(Hpn0/Hpn1,p);if(vptor>1,S0=0;S1=0;S2=0;
pN=p*p^ex;nu=(f-1)/3;fn=pN*f;z=znprimroot(f);
zz=lift(z);t=lift(Mod((1-zz)/f,2*p));c=zz+t*f;
for(a=1,fn/2,if(gcd(a,fn)!=1,next);asurc=lift(a*Mod(c,fn)^-1);
lambda=(asurc*c-a)/fn;u=Mod(lambda*a^-1,pN);
a0=lift((a*z^0)^nu);a1=lift((a*z^2)^nu);a2=lift((a*z)^nu);
if(a0==1,S0=S0+u);if(a1==1,S1=S1+u);if(a2==1,S2=S2+u));
L0=lift(S0);L1=lift(S1);L2=lift(S2);
j=Mod(y,y^2+y+1);Y=L0+j*L1+j^2*L2;nj=valuation(norm(Y),p);
print(f," ",P," vptor=",vptor," T_K=",L," A= ",L0," ",L1," ",L2," ",nj)))))}
\end{verbatim}

\normalsize\smallskip
Let's give a partial table for $p=7$ and $13$, in which 
${\sf vptor} :=v_p(\order {\mathcal T}_K)$ (examples limited to ${\sf vptor} \geq 2$),
and ${\sf nj} = v_p \big({\rm N}_{\Q(j)/\Q}(\Lambda_0+ \Lambda_1 \cdot j  
+ \Lambda_2 \cdot  j^2) \big)$; one sees that, as expected, all the 
examples give ${\sf nj=vptor}$ since ${\mathcal T}_K$ is a finite
$\Z_7[j]$-module which may be decomposed with two $7$-adic characters:

\smallskip
\footnotesize
\begin{verbatim}
f      P                      vptor    T_K         coefficients       nj
313    x^3+x^2-104*x+371        2      [7,7]       [41, 41, 48]       2
577    x^3+x^2-192*x+171        2      [49]        [183, 17, 280]     2
823    x^3+x^2-274*x+61         3      [343]       [761, 419, 437]    3
883    x^3+x^2-294*x+1439       2      [7,7]       [14, 0, 35]        2
1051   x^3+x^2-350*x-2608       2      [49]        [4, 247, 309]      2
1117   x^3+x^2-372*x+2565       2      [7,7]       [7, 7, 42]         2
1213   x^3+x^2-404*x+629        2      [49]        [45, 313, 268]     2
1231   x^3+x^2-410*x-1003       2      [49]        [247, 73, 273]     2
1237   x^3+x^2-412*x+1741       2      [49]        [108, 336, 128]    2
1297   x^3+x^2-432*x-1345       2      [49]        [277, 62, 14]      2
1327   x^3+x^2-442*x-344        2      [49]        [217, 340, 251]    2
1381   x^3+x^2-460*x-1739       4      [343,7]     [1738, 2186, 2361] 4
1567   x^3+x^2-522*x-4759       2      [49]        [219, 137, 78]     2
(...)
2203   x^3+x^2-734*x+408        2      [7,7]       [28, 28, 35]       2
2251   x^3+x^2-750*x-1584       2      [49]        [191, 274, 151]    2
2557   x^3+x^2-852*x+9281       3      [49,7]      [235, 3, 286]      3
\end{verbatim}
\normalsize

\smallskip
For $f= 33199$, $P = x^3 + x^2 - 11066\,x + 238541$,
${\mathcal T}_K \simeq \Z/7\Z \times \Z/7\Z$, $h=14$, and
the annihilator is equivalent, modulo $1+ \sigma + \sigma^2$, 
to $A=\sigma-2$.

\smallskip
For $f= 20857$, $P=x^3 + x^2 - 6952\,x + 210115$,
${\mathcal T}_K \simeq \Z/7^2\Z \times \Z/7^2\Z$, $h=1$, and the
annihilator is equivalent to $A=7^2 (\sigma-3)$ where $\sigma-3$
is invertible modulo $7$. 

\smallskip
For $f= 1381$, ${\mathcal T}_K \simeq \Z/7^3\Z \times \Z/7\Z$,
$h=1$, $A= 1738 + 2186\,\sigma + 2361\, \sigma^2$ is equivalent 
to $7 \cdot (448 + 623\,\sigma)$ and
$448 + 623\,\sigma$ operates on ${\mathcal T}_K^7 \simeq \Z/7^2\Z$
as $\sigma - 18$ modulo $7^2$ where $18$ is of order $3$ modulo $7^2$
as expected.

\smallskip
For $f= 39679$, ${\mathcal T}_K \simeq \Z/7^3\Z \times \Z/7\Z \times \Z/7\Z$,
$h=7$, and one finds  the annihilator $A=7^2 (\sigma-4)$ where 
$\sigma-4$ is not invertible (${\rm N}_{\Q(j)/\Q}(j-4)=21$).

\medskip
For $p=13$, the same program gives the following similar results:

\footnotesize
\begin{verbatim}
f      P                     vptor    T_K          coefficients       nj
1033   x^3+x^2-344*x+1913      2      [169]        [311, 455, 919]    2
1459   x^3+x^2-486*x+2864      2      [13,13]      [101, 88, 153]     2
1483   x^3+x^2-494*x-2197      2      [169]        [911, 1868, 1628]  2
1543   x^3+x^2-514*x+4229      2      [169]        [1598, 603, 1866]  2
1747   x^3+x^2-582*x-4141      2      [169]        [1952, 505, 155]   2
3391   x^3+x^2-1130*x+14192    3      [169,13]     [803, 1765, 283]   3
4423   x^3+x^2-1474*x+10648    2      [169]        [52, 1213, 1888]   2
4933   x^3+x^2-1644*x-1827     2      [13,13]      [92, 79, 105]      2
5011   x^3+x^2-1670*x-4083     2      [169]        [602, 1673, 869]   2
5479   x^3+x^2-1826*x+13799    2      [13,13]      [93, 158, 28]      2
7321   x^3+x^2-2440*x-45824    2      [169]        [745, 409, 1546]   2
7963   x^3+x^2-2654*x+43944    2      [169]        [1805, 794, 860]   2
9319   x^3+x^2-3106*x-67649    2      [13,13]      [26, 52, 0]        2
\end{verbatim}
\normalsize

\section{Experiments and heuristics about the case $p=2$}

Conjecture \ref{conj} gives various possibilities of annihilation,
depending on the choice of ${\mathcal A}_{K,n}(c)$, ${\mathcal A}'_{K,n}(c)$ 
or else, and of the degree of $K/\Q$, odd, even, or a $2$th power.
We shall give some illustrations with quadratic, quartic and cubic fields.

\subsection{Quadratic fields}
Although the order of ${\mathcal T}_K$ is known
and given by $\frac{1}{2} L_2(1,\chi)$ (for $K\ne \Q(\sqrt 2)$), 
we give the computations
for the quadratic fields $K$ of conductor $f \geq 5$ with 
${\mathcal A}'_{K,n}(c)$ ($a \in [1, f_n/2]$) instead of 
${\mathcal A}_{K,n}(c)$ to test the conjecture; 
the computation of the Artin symbols is easily 
given by PARI with ${\sf kronecker(f,a)}=\pm1$.
The modulus $f_n = {\rm l.c.m.}\,(f_K, 4\cdot 2^n)$ is 
computed exactely and we take $n=e+2$.

\smallskip
From the annihilator $A' = a_0 + a_1 \cdot \sigma$ (in $({\sf L_0, L_1})$), 
we deduce, modulo the norm, an equivalent annihilator 
denoted by abuse $A' = a_1 - a_0 \in \Z$.

\smallskip
One finds $A' \equiv 2\cdot \order {\mathcal T}_K \pmod {2^{2+e}}$ 
for all $f \ne 8$ (only case with $K \cap \Q_\infty \ne \Q$) in this 
interval; then the class group is given
(be careful to take ${\sf nt}$ large enought for the computation 
of the structure of ${\mathcal T}_K$):

\footnotesize
\begin{verbatim}
{p=2;nt=18;bf=5;Bf=10^4;for(f=bf,Bf,v=valuation(f,2);M=f/2^v;
if(core(M)!=M,next);if((v==1||v>3)||(v==0 & Mod(M,4)!=1)||
(v==2 & Mod(M,4)==1),next);P=x^2-f;K=bnfinit(P,1);Kpn=bnrinit(K,p^nt);
C5=component(Kpn,5);Hpn0=component(C5,1);Hpn=component(C5,2);
h=component(component(component(K,8),1),2);L=List;ex=0;
i=component(matsize(Hpn),2);for(k=1,i-1,co=component(Hpn,i-k+1);
if(Mod(co,p)==0,val=valuation(co,p);if(val>ex,ex=val);
listinsert(L,p^val,1)));Hpn1=component(Hpn,1);
vptor=valuation(Hpn0/Hpn1,p);tor=p^vptor;S0=0;S1=0;w=valuation(f,p);
pN=p^2*p^ex;fn=pN*f/2^w;if(ex==0 & w==3,fn=p*fn);
for(cc=2,10^2,if(gcd(cc,p*f)!=1 || kronecker(f,cc)!=-1,next);c=cc;break);
for(a=1,fn/2,if(gcd(a,fn)!=1,next);asurc=lift(a*Mod(c,fn)^-1);
lambda=(asurc*c-a)/fn;u=Mod(lambda*a^-1,pN);
s=kronecker(f,a);if(s==1,S0=S0+u);if(s==-1,S1=S1+u));
L0=lift(S0);L1=lift(S1);A=L1-L0;if(A!=0,A=p^valuation(A,p));
print(f," P=",P," ",L0," ",L1," A=",A," tor=",tor," T_K=",L," Cl_K=",h))}

f_K=8     P=x^2-8     (1,0)        A'=1    tor=1     T_K=[]      Cl_K=[]
(...)
f_K=508   P=x^2-508   (223,479)    A'=256  tor=128   T_K=[128]   Cl_K=[]
(...)
f_K=1160  P=x^2-1160  (2,6)        A'=4    tor=2     T_K=[2]     Cl_K=[2,2]
f_K=1164  P=x^2-1164  (12,4)       A'=8    tor=4     T_K=[4]     Cl_K=[4]
(...)
f_K=1185  P=x^2-1185  (1640,1640)  A'=0    tor=1024  T_K=[2,512] Cl_K=[2]
f_K=1189  P=x^2-1189  (2,6)        A'=4    tor=2     T_K=[2]     Cl_K=[2]
(...)
f_K=1196  P=x^2-1196  (4,20)       A'=16   tor=8     T_K=[8]     Cl_K=[2]
f_K=1201  P=x^2-1201  (7752,3656)  A'=4096 tor=2048  T_K=[2048]  Cl_K=[]
(...)
f_K=1209  P=x^2-1209  (4,4)        A'=0    tor=4     T_K=[2,2]   Cl_K=[2]
(...)
f_K=1217  P=x^2-1217  (16,48)      A'=32   tor=16    T_K=[16]    Cl_K=[]
f_K=1221  P=x^2-1221  (8,8)        A'=0    tor=8     T_K=[2,4]   Cl_K=[4]
(...)
f_K=1596  P=x^2-1596  (16,16)      A'=0    tor=16    T_K=[8, 2]  Cl_K=[4,2]
\end{verbatim} 
\normalsize

\begin{remarks}
(i) For $f=1160$, one sees that $\order \Cl_K^\infty = \frac{1}{2}\order \Cl_K$
(indeed, $-1$ is norm in $K/\Q$, cf. \eqref{clinfty}).

\smallskip
(ii) It seems that for all the conductors, $A'$ is of the form 
$2^h\,(1+\sigma)$ up to a $2$-adic unit, where $h\geq 0$
takes any value and can exceed the exponent.

\smallskip
(iii) For $f$ prime, the annihilator of ${\mathcal T}_K$, given by the Theorem 
\ref{thmp}, or by any Solomon's type element, is related to its order:

\smallskip
\centerline{$\Frac{1}{2}\,L_2(1,\chi) \sim \Frac{1}{2}\,\sum_{a=1}^f
\chi(a)\cdot {\rm log}(1 - \zeta_f^a) = \Frac{1}{2} \cdot 
\big [{\rm log}(\eta_K^{}) - {\rm log}(\eta_K^\sigma) \big]$,}

\smallskip
where $\eta_K^{} = {\rm N}_{\Q^f/K}(1 - \zeta_f)$ (here
the character $\chi$ is primitive modulo $f$ since $K= k_\chi$).
The following program verifies (at least for these kind of prime conductors
with trivial class group) that we have $\eta_K^{} \cdot \varepsilon = 
\pm \sqrt{f}$, where $\varepsilon$ is the fundamental unit of $K$ 
or its inverse (the program gives in ${\sf N_0}$ and ${\sf N_1}$ the conjugates 
of $\eta_K^{}$ and gives $\varepsilon$ in ${\sf E}$):

\smallskip\footnotesize
\begin{verbatim}
{f=1201;N0=1;N1=1;X=exp(2*I*Pi/f);z=znprimroot(f);E=quadunit(f);zk=1;
for(k=1,(f-1)/2,zk=zk*z^2;N0=N0*(1-X^lift(zk));N1=N1*(1-X^lift(zk*z)));
print(N0*E," ",N1/E)}
\end{verbatim} 
\normalsize

We find $N_0\,\varepsilon = N_1\,\varepsilon^{-1} \approx
34.65544690 = \sqrt{1201}$, which implies that:
$$\Frac{1}{2}\,L_2(1,\chi) \sim \Frac{1}{2}\,(2\, {\rm log}(\varepsilon))
= {\rm log}(\varepsilon). $$

A direct computation gives ${\rm log}(\varepsilon) \sim 2^{12}$ as expected 
since $\order {\mathcal T}_K = 2^{11}$ with $\order {\mathcal R}_K \sim 2^{10}$
\cite[Proposition 5.2]{Gr2} and $\order {\mathcal W}_K= 2$ since $2$ splits in $K$.
Same kind of result with $f=1217$.
\end{remarks}

\subsection{A familly of cyclic quartic fields of composite conductor}

We consider a conductor $f$ product of two prime numbers $q_1$ and $q_2$
such that $q_1-1\equiv 2 \pmod 4$ and $q_2-1 \equiv 0 \pmod 8$. 
So there exists only one real cyclic quartic field $K$ of conductor $f$ which is
found eliminating the imaginary and non-cyclic fields;
the quadratic subfield of $K$ is $k = \Q(\sqrt{q_2})$. The program is written with 
${\mathcal A}'_{K,n}(c)$ and gives all information for $k$ and $K$.

\smallskip
The following result may help to precise the annihilations 
(see \cite[Theorem 2]{Gr7} or \cite[Theorem IV.3.3, Exercise IV.3.3.1]{Gr1}):

\begin{lemma} \label{inv}
Let $k$ be a totally real number field and let $K/k$ be a Galois 
$p$-extension with Galois group $g$ of order $p^r$. Then
we have the fixed point formula:
$\order {\mathcal T}_K^g =\order {\mathcal T}_k \cdot p^h$, where
(${\mathfrak l} \nmid p$ being the ramified primes in $K/k$):
$$h :=  {\rm min\,}(n^{}_0 + r\, ; \ldots, \nu^{}_{\mathfrak l}
+ \varphi^{}_{\mathfrak l} + \gamma^{}_{\mathfrak l}, \ldots) 
- (n^{}_0 + r) + \sm_{{\mathfrak l}\,\nmid\, p} e^{}_{\mathfrak l}, $$ 
with:

\smallskip
\qquad $p^{\nu_{\mathfrak l}} := $ $p$-part of $q^{-1}{\rm log}(\ell)$,
where ${\mathfrak l}\cap \Z =: \ell \Z$,

\qquad  $p^{\varphi_{\mathfrak l}} := $ $p$-part of the residue
degree of $\ell$ in $K/\Q$,

\qquad  $p^{\gamma_{\mathfrak l}} := $ $p$-part of the number
of prime ideals ${\mathfrak L} \mid {\mathfrak l}$ in $K/k$,

\qquad $p^{e_{\mathfrak l}} := $ $p$-part of the ramification
index of ${\mathfrak l}$ in $K/k$.
\end{lemma}

In such famillies of cyclic quartic fields, 
$h=\sm_{{\mathfrak l}\,\nmid\, p} e^{}_{\mathfrak l}$.

\subsubsection{The program}
In the present familly, $h=2$ (resp. $3$) if $q$ is inert (resp. splits) in $k/\Q$.

\smallskip
\footnotesize
\begin{verbatim}
{p=2;nt=18;forprime(qq=17,100,if(Mod(qq,8)!=1,next);Pk=x^2-qq;
k=bnfinit(Pk,1);kpn=bnrinit(k,p^nt);Hkpn=component(component(kpn,5),2);
Lk=List;i=component(matsize(Hkpn),2);
for(j=1,i-1,C=component(Hkpn,i-j+1);if(Mod(C,p)==0,
listinsert(Lk,p^valuation(C,p),1)));forprime(q=5,100,
if(valuation(q-1,2)!=2,next);f=q*qq;Q=polsubcyclo(f,4);
for(j=1,7,P=component(Q,j);K=bnfinit(P,1);C7=component(K,7);
S=component(C7,2);D=component(C7,3);
if(Mod(D,f)!=0 || S!=[4,0] || component(polgalois(P),2)!=-1,next);break); 
Cl=component(component(component(K,8),1),2);Kpn=bnrinit(K,p^nt);
C5=component(Kpn,5);Hpn0=component(C5,1);Hpn=component(C5,2);
Hpn=component(component(Kpn,5),2);L=List;ex=0;
i=component(matsize(Hpn),2);for(k=1,i-1,co=component(Hpn,i-k+1);
if(Mod(co,p)==0,val=valuation(co,p);if(val>ex,ex=val);
listinsert(L,p^val,1)));Hpn1=component(Hpn,1);
vptor=valuation(Hpn0/Hpn1,p);if(vptor>0,S0=0;S1=0;S2=0;S3=0;
pN=p^2*p^ex;fn=pN*f;dqq=(qq-1)/4;dq=(q-1)/2;
z=znprimroot(q);zz=znprimroot(qq);for(cc=3,f,if(gcd(cc,p*f)!=1,next);
cz=lift((cc*z)^dq);czz=lift((cc*zz)^dqq);if(cz!=1 || czz!=1,next);
c=cc;break);cm1=Mod(c,fn)^-1;for(a=1,fn/2,if(gcd(a,fn)!=1,next); 
asurc=lift(a*cm1);lambda=(asurc*c-a)/fn;u=Mod(lambda*a^-1,pN); 
aqq0=lift((a*zz^0)^dqq);aqq1=lift((a*zz^1)^dqq);
aqq2=lift((a*zz^2)^dqq);aqq3=lift((a*zz^3)^dqq);
aq0=lift((a*z^0)^dq);aq1=lift((a*z^1)^dq);
if(aqq0==1 & aq0==1,S0=S0+u);if(aqq0==1 & aq1==1,S2=S2+u);
if(aqq1==1 & aq0==1,S1=S1+u);if(aqq1==1 & aq1==1,S3=S3+u);
if(aqq2==1 & aq0==1,S2=S2+u);if(aqq2==1 & aq1==1,S0=S0+u);
if(aqq3==1 & aq0==1,S3=S3+u);if(aqq3==1 & aq1==1,S1=S1+u));
L0=lift(S0);L1=lift(S1);L2=lift(S2);L3=lift(S3);Y=Mod(y,y^2+1); 
ni=L0+Y*L1+Y^2*L2+Y^3*L3;Nni=valuation(norm(ni),2));V0=1;V1=1;V2=1;V3=1;
if(L0!=0,V0=2^valuation(L0,2));if(L1!=0,V1=2^valuation(L1,2));
if(L2!=0,V2=2^valuation(L2,2));if(L3!=0,V3=2^valuation(L3,2));
print();F=component(factor(f),1);
print("f=",F," Cl=",Cl," P=",P," tor=",2^vptor," Nni=",2^Nni);
print("A=",V0,"*",L0/V0," ",V1,"*",L1/V1," ",V2,"*",L2/V2," ",V3,"*",L3/V3);
print("q=",q," qq=",qq," T_k=",Lk," T_K=",L)))}

f=[5, 17]   h=[2] P=x^4-x^3-23*x^2+x+86 tor=16 Nni=16
  A=[2*5, 4*1, 2*1, 1*0]   q=5 qq=17  T_k=List([2])  T_K=[4, 2, 2]
f=[13, 17]   h=[2] P=x^4-x^3-57*x^2+x+664 tor=32 Nni=32
  A=[2*1, 2*1, 2*3, 2*3]   q=13 qq=17  T_k=[2]  T_K=[4, 4, 2]
f=[17, 29]   h=[2] P=x^4-x^3-125*x^2+x+3452 tor=16 Nni=16
  A=[4*3, 2*1, 1*0, 2*1]   q=29 qq=17  T_k=[2]  T_K=[4, 2, 2]
f=[17, 37]   h=[10] P=x^4-x^3-159*x^2+x+5662 tor=16 Nni=16
  A=[4*1, 2*3, 8*1, 2*7]   q=37 qq=17  T_k=[2]  T_K=[4, 2, 2]
f=[17, 53]   h=[2, 2] P=x^4-x^3-227*x^2+x+11714 tor=32 Nni=32
  A=[2*1, 2*5, 2*3, 2*7]   q=53 qq=17  T_k=[2]  T_K=[4, 4, 2]
f=[17, 61]   h=[2] P=x^4-x^3-261*x^2+x+15556 tor=16 Nni=16
  A=[2*1, 8*1, 2*5, 4*3]   q=61 qq=17  T_k=[2]  T_K=[4, 2, 2]
f=[5, 41]   h=[2] P=x^4-x^3-56*x^2-100*x+160 tor=256 Nni=32
  A=[2*13, 2*45, 2*59, 2*27]   q=5 qq=41  T_k=[16]  T_K=[32, 4, 2]
f=[13, 41]   h=[2] P=x^4-x^3-138*x^2-264*x+1472 tor=256 Nni=32
  A=[2*13, 2*27, 2*51, 2*5]   q=13 qq=41  T_k=[16]  T_K=[32, 4, 2]
f=[29, 41]   h=[2] P=x^4-x^3-302*x^2-592*x+8032 tor=1024 Nni=128
  A=[4*21, 4*5, 4*15, 4*15]   q=29 qq=41  T_k=[16]  T_K=[32, 8, 4]
f=[37, 41]   h=[2] P=x^4-x^3-384*x^2-756*x+13280 tor=256 Nni=32
  A=[2*57, 2*7, 2*47, 2*33]   q=37 qq=41  T_k=[16]  T_K=[32, 4, 2]
f=[41, 53]   h=[2] P=x^4-x^3-548*x^2-1084*x+27712 tor=512 Nni=64
  A=[4*23, 8*15, 4*5, 8*7]   q=53 qq=41  T_k=[16]  T_K=[32, 4, 4]
f=[41, 61]   h=[2, 2] P=x^4-x^3-630*x^2-1248*x+36896 tor=8192 Nni=1024
  A=[32*3, 16*7, 1*0, 16*7]   q=61 qq=41  T_k=[16]  T_K=[32, 16, 16]
f=[5, 73]   h=[2] P=x^4-x^3-100*x^2+187*x+1389 tor=8 Nni=8
  A=[1*5, 1*9, 1*15, 1*3]   q=5 qq=73  T_k=[2]  T_K=[4, 2]
f=[13, 73]   h=[2] P=x^4-x^3-246*x^2+479*x+11171 tor=8 Nni=8
  A=[1*7, 1*13, 1*13, 1*15]   q=13 qq=73  T_k=[2]  T_K=[4, 2]
f=[29, 73]   h=[2] P=x^4-x^3-538*x^2+1063*x+58767 tor=8 Nni=8
  A=[1*5, 1*7, 1*15, 1*5]   q=29 qq=73  T_k=[2]  T_K=[4, 2]
f=[37, 73]   h=[2] P=x^4-x^3-684*x^2+1355*x+96581 tor=128 Nni=128
  A=[1*0, 16*1, 8*1, 8*1]   q=37 qq=73  T_k=[2]  T_K=[8, 8, 2]
f=[53, 73]   h=[10] P=x^4-x^3-976*x^2+1939*x+200241 tor=8 Nni=8
  A=[1*15, 1*15, 1*5, 1*13]   q=53 qq=73  T_k=[2]  T_K=[4, 2]
f=[61, 73]   h=[2] P=x^4-x^3-1122*x^2+2231*x+266087 tor=16 Nni=16
  A=[8*1, 2*3, 1*0, 2*1]   q=61 qq=73  T_k=[2]  T_K=[4, 2, 2]
f=[5, 89]   h=[2, 2] P=x^4-x^3-122*x^2-217*x+1699 tor=16 Nni=16
  A=[1*0, 2*1, 8*1, 2*3]   q=5 qq=89  T_k=[2]  T_K=[4, 2, 2]
f=[13, 89]   h=[2] P=x^4-x^3-300*x^2-573*x+13625 tor=8 Nni=8
  A=[1*1, 1*7, 1*11, 1*13]   q=13 qq=89  T_k=[2]  T_K=[4, 2]
f=[29, 89]   h=[2] P=x^4-x^3-656*x^2-1285*x+71653 tor=8 Nni=8
  A=[1*11, 1*5, 1*1, 1*15]   q=29 qq=89  T_k=[2]  T_K=[4, 2]
f=[37, 89]   h=[2] P=x^4-x^3-834*x^2-1641*x+117755 tor=8 Nni=8
  A=[1*9, 1*15, 1*3, 1*5]   q=37 qq=89  T_k=[2]  T_K=[4, 2]
f=[53, 89]   h=[2] P=x^4-x^3-1190*x^2-2353*x+244135 tor=16 Nni=16
  A=[4*1, 2*5, 4*1, 2*7]   q=53 qq=89  T_k=[2]  T_K=[4, 2, 2]
f=[61, 89]   h=[2] P=x^4-x^3-1368*x^2-2709*x+324413 tor=8 Nni=8
  A=[1*1, 1*9, 1*11, 1*11]   q=61 qq=89  T_k=[2]  T_K=[4, 2]
f=[5, 97]   h=[2] P=x^4-x^3-133*x^2-479*x+36 tor=16 Nni=16
  A=[2*5, 8*1, 2*1, 4*3]   q=5 qq=97  T_k=[2]  T_K=[4, 2, 2]
f=[13, 97]   h=[10] P=x^4-x^3-327*x^2-1255*x+2558 tor=16 Nni=16
  A=[4*1, 2*7, 8*1, 2*3]   q=13 qq=97  T_k=[2]  T_K=[4, 2, 2]
f=[29, 97]   h=[2] P=x^4-x^3-715*x^2-2807*x+16914 tor=16 Nni=16
  A=[2*3, 8*1, 2*3, 4*3]   q=29 qq=97  T_k=[2]  T_K=[4, 2, 2]
f=[37, 97]   h=[2] P=x^4-x^3-909*x^2-3583*x+28748 tor=16 Nni=16
  A=[4*3, 2*7, 1*0, 2*3]   q=37 qq=97  T_k=[2]  T_K=[4, 2, 2]
f=[53, 97]   h=[2] P=x^4-x^3-1297*x^2-5135*x+61728 tor=64 Nni=64
  A=[8*3, 4*7, 16*1, 4*7]   q=53 qq=97  T_k=[2]  T_K=[8, 4, 2]
f=[61, 97]   h=[2] P=x^4-x^3-1491*x^2-5911*x+82874 tor=32 Nni=32
  A=[2*7, 2*5, 2*5, 2*7]   q=61 qq=97  T_k=[2]  T_K=[4, 4, 2]
\end{verbatim} 
\normalsize

\subsubsection{The case $f=5\cdot73$}
One may try to find a contradiction to Conjecture \ref{conj} 
with the ${\mathcal A}'_{K,n}(c)$ given by the above data.
One sees that $\frac{1}{2}{\mathcal A}'_{K,n}(c)$ 
is not always in $\Z[G_K]$, but modulo the norm 
we have an annihilator of the form $2\cdot {\mathcal A}''_{K,n}(c)$, 
and similarly we may ask under what condition 
${\mathcal A}''_{K,n}(c)$ annihilates ${\mathcal T}_K$. 

\smallskip
For $f=5\cdot73$, $P=x^4-x^3-100\,x^2+187\,x+1389$, for which
we have ${\mathcal T}_K \simeq \Z/4\Z \times \Z/2 \Z$, 
${\mathcal T}_k \simeq \Z/2 \Z$, ${\sf Cl}=2$,
${\mathcal A}'_{K,n}(c)= 5 +9\,\sigma + 15\,\sigma^2 + 3\, \sigma^3$,
giving:
$${\mathcal A}''_{K,n}(c) = \hbox{$\frac{1}{2}$} 
\big[5 +9\,\sigma + 15\,\sigma^2 + 3\, \sigma^3
- 3\,(1 +\sigma + \sigma^2 + \sigma^3)\big] \equiv 1 - \sigma + 2\,\sigma^2 \!\! \pmod 4$$
without obvious contradiction since $\order {\mathcal T}_K^g=8$
(i.e.,  ${\mathcal T}_K^g =  {\mathcal T}_K$)
and $\order {\mathcal T}_K^{G_K}=4$ (Lemma \ref{inv}).
Moreover, we deduce from this that ${\rm N}_{K/k}({\mathcal T}_K)={\mathcal T}_k$.

\subsection{Cyclic cubic fields of prime conductors}
The following program gives, for $p=2$ and for cyclic cubic fields of prime 
conductor ${\sf f}$, the group structure of ${\mathcal T}_K$ in ${\sf L}$
(from \cite[\S\,3.2]{Gr3}; recall that in all such
programs, the parameter ${\sf nt}$ must be large enough regarding the 
exponent of ${\mathcal T}_K$), then the (conjectural) annihilator 
${\mathcal A}'_{K,n}(c)$, reduced modulo $1+\sigma+\sigma^2$; 
it is equal, up to an invertible element, to a power of $2$ ($2$
is inert in $\Q(j)$):

\smallskip\footnotesize
\begin{verbatim}
{p=2;nt=12;forprime(f=10^4,2*10^4,if(Mod(f,3)!=1,next);P=polsubcyclo(f,3);
K=bnfinit(P,1);Kpn=bnrinit(K,p^nt);C5=component(Kpn,5);
Hpn0=component(C5,1);Hpn=component(C5,2);L=List;ex=0;
i=component(matsize(Hpn),2);for(k=1,i-1,co=component(Hpn,i-k+1);
if(Mod(co,p)==0,val=valuation(co,p);if(val>ex,ex=val);
listinsert(L,p^val,1)));Hpn1=component(Hpn,1);
vptor=valuation(Hpn0/Hpn1,p);if(vptor>2,S0=0;S1=0;S2=0;pN=p^2*p^ex;
D=(f-1)/3;fn=pN*f;z=znprimroot(f);zz=lift(z);t=lift(Mod((1-zz)/f,p));
c=zz+t*f;for(a=1,fn/2,if(gcd(a,fn)!=1,next);asurc=lift(a*Mod(c,fn)^-1);
lambda=(asurc*c-a)/fn;u=Mod(lambda*a^-1,pN);
a0=lift((a*z^0)^D);a1=lift((a*z^2)^D);a2=lift((a*z)^D);
if(a0==1,S0=S0+u);if(a1==1,S1=S1+u);if(a2==1,S2=S2+u));
L0=lift(S0);L1=lift(S1);L2=lift(S2);L1=L1-L0;L2=L2-L0;
A=gcd(L1,L2);A=2^valuation(A,2);print(f," ",P," ", A," ",L)))}

f       P                          A    L
10399   x^3+x^2-3466*x+7703        4    [4,4]
10513   x^3+x^2-3504*x-80989       8    [8,8]
10753   x^3+x^2-3584*x-76864       4    [4,4]
10771   x^3+x^2-3590*x-26728       4    [4,4]
10903   x^3+x^2-3634*x+26248       8    [8,8]
10939   x^3+x^2-3646*x-46592      16    [16,16]
10957   x^3+x^2-3652*x-39364       4    [4,4]
11149   x^3+x^2-3716*x+39228       4    [2,2,2,2]
(...)
12757   x^3+x^2-4252*x+103001      4    [4,4]
13267   x^3+x^2-4422*x+96800      16    [16,16]
13297   x^3+x^2-4432*x+94064       4    [4,4]
13309   x^3+x^2-4436*x+100064      4    [4,4]
13591   x^3+x^2-4530*x-63928       8    [8,8]
\end{verbatim}

\normalsize
\subsection{Cyclic quartic fields of prime conductors}

Let's give the same program for prime conductors
$f \equiv 1 \!\!\pmod 8$, with the annihilator ${\mathcal A}_{K,n}(c)$:

\smallskip\footnotesize
\begin{verbatim}
{p=2;nt=18;d=4;forprime(f=5,500,if(Mod(f,2*d)!=1,next);P=polsubcyclo(f,d);
K=bnfinit(P,1);Kpn=bnrinit(K,p^nt);C5=component(Kpn,5);Hpn0=component(C5,1);
Hpn=component(C5,2);L=List;ex=0;
i=component(matsize(Hpn),2);for(k=1,i-1,co=component(Hpn,i-k+1);
if(Mod(co,p)==0,val=valuation(co,p);if(val>ex,ex=val);
listinsert(L,p^val,1)));Hpn1=component(Hpn,1);
vptor=valuation(Hpn0/Hpn1,p);if(vptor>1,S0=0;S1=0;S2=0;S3=0;
pN=p^2*p^ex;D=(f-1)/d;fn=pN*f;z=znprimroot(f);zz=lift(z);
t=lift(Mod((1-zz)/f,p));c=zz+t*f;for(a=1,fn,if(gcd(a,fn)!=1,next);
asurc=lift(a*Mod(c,fn)^-1);lambda=(asurc*c-a)/fn;u=Mod(lambda*a^-1,pN);
a0=lift((a*z^0)^D);a1=lift((a*z^1)^D);a2=lift((a*z^2)^D);a3=lift((a*z^3)^D);
if(a0==1,S0=S0+u);if(a1==1,S1=S1+u);if(a2==1,S2=S2+u);if(a3==1,S3=S3+u));
L0=lift(S0);L1=lift(S1);L2=lift(S2);L3=lift(S3);Y=Mod(y,y^2+1);
ni=L0+Y*L1+Y^2*L2+Y^3*L3;Nni=valuation(norm(ni),2);
print(f," ",P," ",L0," ",L1," ",L2," ",L3," ",L," ",2^Nni)))}
\end{verbatim}

\normalsize
One gets the following examples (with ${\sf vptor}>1$ and where 
${\sf 2^{Nni}}$ is the norm of $L_0 - L_2 + (L_1- L_3) \sqrt{-1}$ with 
${\mathcal A}_{K,n}(c) = L_0+L_1 \sigma + L_2 \sigma^2+ L_3 \sigma^3$,
given in ${\sf A=[L0, L1, L2, L3]}$); then the list ${\sf L}$ gives the structure
of ${\mathcal T}_K$:
\footnotesize
\begin{verbatim}
f     P                             A                    L         2^Nni
17    x^4+x^3-6*x^2-x+1             [4, 6, 0, 6]         [4]       16
41    x^4+x^3-15*x^2+18*x-4         [90, 28, 102, 100]   [32]      16
73    x^4+x^3-27*x^2-41*x+2         [4, 4, 0, 0]         [2,2,2]   32
89    x^4+x^3-33*x^2+39*x+8         [4, 4, 0, 0]         [2,2,2]   32
97    x^4+x^3-36*x^2+91*x-61        [8, 10, 12, 2]       [4]       16
113   x^4+x^3-42*x^2-120*x-64       [16, 28, 8, 12]      [2,2,8]   64
137   x^4+x^3-51*x^2-214*x-236      [26, 8, 30, 16]      [16]      16
193   x^4+x^3-72*x^2-205*x-49       [6, 0, 14, 12]       [4]       16
233   x^4+x^3-87*x^2+335*x-314      [4, 0, 0, 4]         [2,2,2]   32
241   x^4+x^3-90*x^2-497*x-739      [6, 0, 6, 4]         [4]       16
257   x^4+x^3-96*x^2-16*x+256       [28, 20, 20, 60]     [2,4,16]  128
281   x^4+x^3-105*x^2+123*x+236     [4, 4, 0, 0]         [2,2,2]   32
313   x^4+x^3-117*x^2+450*x-324     [78, 12, 106, 108]   [32]      16
337   x^4+x^3-126*x^2+316*x+104     [28, 12, 28, 28]     [2,8,8]   256
353   x^4+x^3-132*x^2+684*x-928     [112, 60, 80, 68]    [2,2,32]  64
401   x^4+x^3-150*x^2-25*x+625      [14, 4, 6, 8]        [4]       16
409   x^4+x^3-153*x^2-230*x+548     [22, 8, 26, 24]      [8]       16
433   x^4+x^3-162*x^2+839*x-1003    [2, 4, 10, 0]        [4]       16
449   x^4+x^3-168*x^2-477*x+335     [10, 4, 10, 8]       [4]       16
457   x^4+x^3-171*x^2+1114*x-2044   [76, 10, 28, 30]     [32]      16
\end{verbatim}

\normalsize
\subsection{Detailed example of annihilation}

The case of the cyclic quartic field $K$ of conductor
$f=3433$ is particularly interesting:

\subsubsection{Numerical data}
We have $P=x^4 + x^3 - 1287\, x^2 - 12230 \,x + 3956$ and
${\mathcal T}_K \simeq \Z/2^7\Z$, knowing that the quadratic
subfield $k = \Q(\sqrt{3433})$ is such that ${\mathcal T}_k \simeq \Z/2^6\Z$:

\smallskip\footnotesize
\begin{verbatim}
{P=x^4+x^3-1287*x^2-12230*x+3956;K=bnfinit(P,1);p=2;nt=18;
Kpn=bnrinit(K,p^nt);Hpn=component(component(Kpn,5),2);L=List;
i=component(matsize(Hpn),2);for(k=1,i-1,c=component(Hpn,i-k+1);
if(Mod(c,p)==0,listinsert(L,p^valuation(c,p),1)));print("Structure: ",L)}
Structure: List([128])

{P=x^2-3433;K=bnfinit(P,1);p=2;nt=18;Kpn=bnrinit(K,p^nt);
Hpn=component(component(Kpn,5),2);L=List;i=component(matsize(Hpn),2);
for(k=1,i-1,c=component(Hpn,i-k+1);if(Mod(c,p)==0,
listinsert(L,p^valuation(c,p),1)));print("Structure: ",L)}
Structure: List([64])
\end{verbatim}

\normalsize \smallskip
The class group of $K$ is trivial and its three fundamental units are:

\footnotesize
\begin{verbatim}
[227193/338*x^3-6613325/338*x^2-93274465/338*x+14925255/169,
34349/169*x^3+1388772/169*x^2+10559389/169*x-3491425/169,
70276336974818125/338*x^3-677429229869394661/338*x^2
       -83238272983560888143/338*x+13065197272033438434/169]
\end{verbatim}

\normalsize
\subsubsection{Annihilation from ${\mathcal A}_{K,n}(c)$}
We have computed ${\mathcal A}_{K,n}(c)$ and obtained:
$${\mathcal A}_{K,n}(c) =:A_K \equiv 
 8 \cdot 13+ 2 \cdot 21\, \sigma + 16 \cdot 7 \,\sigma^2
+ 2 \cdot 23 \,\sigma^3 \pmod {2^7}. $$

Let $h$ be a group generator of ${\mathcal T}_K$ (order $2^7$) and let $h_0$ 
be a generator of ${\mathcal T}_k$ (order $2^6$); it is easy to prove that 
one may suppose $h^2=j_{K/k}(h_0)$ (injectivity of the transfer map $j_{K/k}$) 
and $h_0^{\sigma^2}=h_0$. We put $j_{K/k}(h_0) =: h_0$ for simplicity.
Then it follows that 
$$h^{A_K}=h_0^{4\cdot 13+ 21 \, \sigma + 8 \cdot 7 \, \sigma^2+ 23 \, \sigma^3}=1. $$

Since $h_0^{\sigma^2} = h_0$, we obtain
$h^{A_K}=h_0^{(4 \cdot 13+ 8 \cdot 7)+ (21 + 23) \,\ov \sigma}=
h_0^{4 \cdot 27+ 4\cdot 11 \,\ov \sigma}=1$;
giving, modulo the norm $1+\ov \sigma$,
$h_0^{4\cdot(27-11)} = h_0^{2^6}=1$, as expected.

\subsubsection{Annihilation from ${\mathcal A}'_{K,n}(c)$}
There is (by accident ?) no numerical obstruction for an annihilation
by $A'_K := {\mathcal A}'_{K,n}(c)$, with the same program
replacing ``${\sf for(a=1,fn,... )}$'' by ``${\sf for(a=1,fn/2,... )}$''.
Then it follows that the program gives
$h^{A'_K}=h^{4 \cdot 13+ 21 \, \sigma + 
8 \cdot 15 \, \sigma^2+ 23 \, \sigma^3}=1$.
Since the restriction of $A'_K$ to $k$ is $A'_k$ (no Euler factors), we get:
$$h_0^{A'_k}=h_0^{4 \cdot 13 + 8 \cdot 15
+ (21 + 23)  \cdot \ov \sigma} = 
h_0^{4\cdot 43+4\cdot 11\, \ov \sigma}=1$$

which is equivalent, modulo the norm, to the annihilation by
$4\cdot 43 - 4\cdot 11= 2^7$ for a cyclic group of order $2^6$.

\smallskip
Now we may return to the annihilation of $h$; since $h^{1+\sigma^2}
\in j_{K/k}({\mathcal T}_k)$ we put $h^{1+\sigma^2}= h_0^t$.
Then, with $u=13$, $v=21$, $w=15$, $z=23$, we have:
\begin{equation*}
\begin{aligned}
h^{4u+v \sigma +8 w \sigma^2+z \sigma^3} &=
h_0^{2 u +4 w \sigma^2} h^{(v+z \sigma^2) \sigma} \\
&=h_0^{2 u +4 w +23 \, t \,\sigma} h^{(v-z) \sigma}
=h_0^{2 \cdot 43 + 23\, t\, \sigma} h^{-2\, \sigma} \\
&= h_0^{2 \cdot 43 + (23\, t -1)\,\sigma} =
 h_0^{2 \cdot 43 - 23\, t +1} = h_0^{87- 23\, t}=1
\end{aligned}
\end{equation*}
 
 so necessarily $87- 23 \, t \equiv 0 \pmod {2^6}$, giving
$t \equiv 1\pmod {2^6}$. So we can write:
$$h^{1+\sigma^2}=j_{K/k}(h_0). $$

\subsubsection{Direct study of the $G_K$-module structure of
${\mathcal T}_K$}
We consider ${\mathcal T}_K$ only given with the following information:
$h$ is a group generator such that $h^2=h_0$, a generator of 
$j_{K/k}({\mathcal T}_K)$; $h^\sigma = h^x$, $x\in \Z/2^7\Z$,
whence $h_0^\sigma = h_0^x = h_0^{-1}$ giving $x \equiv -1 \pmod{2^6}$.
The relation $h^{\sigma^2+1}=h^{x^2+1}=h^2=h_0$ gives again $t=1$
in the previous notation $h^{\sigma^2+1}=h_0^t$.
Moreover, $h^{\sigma^2-1}=h^{x^2-1}=1$, which is in 
accordance with Lemma \ref{inv} and gives 
${\mathcal T}_K^g = {\mathcal T}_K$.

\smallskip
If we take into account these theoretical informations for the 
``annihilators'' $A_K$ and $A'_K$ we find no contradiction, but
we do not know if $x=-1$ or $x=-1+2^6$ (modulo $2^7$).
The prime $2$ splits in $k$, is inert in $K/k$ and the class number
of $K$ is $1$; so we have ${\mathcal W}_K \simeq {\mathcal W}_k \simeq \Z/2\Z$
and ${\mathcal T}_K = {\rm tor}_{\Z_2}^{} \big(U_K \big / \ov E_K \big)$;
then the result about $x$ depends on the exact sequence \eqref{exseq}:
$$1\to \Z/2\Z  \too {\mathcal T}_K \simeq \Z/2^7\Z
 \mathop {\tooo}^{{\rm log}}  {\rm tor}_{\Z_2}^{}\big({\rm log}\big 
(U_K \big) \big / {\rm log} (\ov E_K) \big) =: {\mathcal R}_K \to 0, $$

knowing the units and then the structure of the regulator ${\mathcal R}_K$.

\subsubsection{About the case $f_K=233$} \label{555}
The field $K$ is defined by the polynomial $P=x^4+x^3-87\,x^2+335\,x-314$
for which ${\mathcal T}_K \simeq (\Z/2\Z)^3$ and ${\mathcal T}_k \simeq \Z/2\Z$.

\smallskip
In this case an annihilator is  $A_K= 4\cdot (1 + \sigma^3)$,
which shows that $A'_K= 2 \cdot (1 + \sigma^3)$ is also suitable.
Then $A''_K = \frac{1}{2} A'_K$ should be equivalent to $1-\sigma$.

\smallskip
Since $2$ splits completely in $K$, we have ${\mathcal T}_K = {\mathcal W}_K
\simeq (\Z/2\Z)^3$ and in the same way, ${\mathcal T}_k = {\mathcal W}_k \simeq \Z/2\Z$, 
for which the Galois structures are well-known: in particular, $1-\sigma$ 
{\it does not annihilate} ${\mathcal T}_K$ (the class of $(1,-1,1,-1)$ is invariant).
Another proof: use Lemma \ref{inv} giving here $\order {\mathcal T}_K^{G_K}=2$.

\section{$p$-adic measures and annihilations}

To establish (in Section \ref{section9}) a link with the values of 
$p$-adic $L$-functions,
$L_p(s, \chi)$, at $s=1$, we shall refer to \cite[Section II]{Gr6} 
using the point of view of explicit $p$-adic measures 
(from pseudo-measures in the sense of \cite{Se}) 
with a Mellin transform for the construction of $L_p(s, \chi)$
and the application to some properties of the $\lambda$ invariants 
of Iwasawa's theory. 

\smallskip
But since we only need the value $L_p(1,\chi)$, instead of $L_p(s,\chi)$,
for $s \in \Z_p$, we can simplify the general setting, using a similar 
computation of ${\mathcal S}_{L_n}(c)^*$, directly in $\Z[G_n]$,
given by Oriat in \cite[Proposition 3.5]{O}.

\subsection{Definition of ${\mathcal A}_{L_n}$ 
and ${\mathcal A}_{L_n}(c)$} \label{defAc}

Let $n \geq n_0+e$, where ${\mathcal T}_K^{p^e}=1$,
and put $\varphi_n := \varphi(q p^n) = (p-1)\cdot p^n$
if $p\ne 2$, $\varphi_n = 2^{n+1}$ otherwise. 

\smallskip
We consider (where $c$ is odd and prime to $f_n$ and where 
$a$ runs trough the integers in $[1,f_n]$, prime to $f_n$):
\begin{equation} \label{defA}
{\mathcal A}_{L_n}\!\! :=\! \Frac{-1}{f_n \varphi_n} 
\sm_a a^{\varphi_n} \Big(\Frac{L_n}{a} \Big)  \ \ \& \ \  {\mathcal A}_{L_n}(c)\! := \!
\Big[1-c^{\varphi_n} \Big(\Frac{L_n}{c} \Big)\Big] {\mathcal A}_{L_n} . 
\end{equation}

For now, these elements, or more precisely their restrictions to $K$,
are not to be confused with the restrictions
${\mathcal A}_{K,n}(c)$ of $ {\mathcal S}_{L_n}(c)^*$ defined in \S\,\ref{slnc},
even we shall prove that they are indeed equal; but such an 
expression is more directly associated to $L_p$-functions. Then:
\begin{equation*}
\begin{aligned}
{\mathcal A}_{L_n}(c) &=
\Big[1-c^{\varphi_n} \Big(\Frac{L_n}{c} \Big)\Big ]  \Frac{-1}{f_n \varphi_n} 
\sm_a a^{\varphi_n} \Big(\Frac{L_n}{a} \Big)  \\
&  \ \wt{\equiv}\  \Frac{-1}{f_n \varphi_n} \Big[ \sm_a a^{\varphi_n} \Big(\Frac{L_n}{a} \Big)
- \sm_a a^{\varphi_n} c^{\varphi_n} 
\Big(\Frac{L_n}{a } \Big) \Big(\Frac{L_n}{ c} \Big)\Big ] \\
& \hbox{\ (in the same way, use $a'_{c}$ such that} \\
&\hbox{\hspace{2.3cm}$a'_{c} \cdot  c \equiv a  \!\! \pmod {f_n}$, $1 \leq a'_{c} \leq f_n$)} \\
&  \ \wt{\equiv}\  \Frac{-1}{f_n \varphi_n} \Big[ \sm_a a^{\varphi_n} \Big(\Frac{L_n}{a} \Big)
- \sm_a a'_{c}{}^{\varphi_n} c^{\varphi_n} 
\Big(\Frac{L_n}{a'_{c}} \Big) \Big(\Frac{L_n}{ c} \Big)\Big ]  \\
&  \ \wt{\equiv}\   \Frac{1}{f_n \varphi_n}  
\sm_a \Big[(a'_{c} \cdot  c)^{\varphi_n} - a^{\varphi_n}\Big]
\Big(\Frac{L_n}{a} \Big) .
\end{aligned}
\end{equation*}

\begin{lemma} 
We have $(a'_{c} \cdot  c)^{\varphi_n} - a^{\varphi_n} \equiv 0 \pmod{f_n \varphi_n}$.
\end{lemma}

\begin{proof}
By definition, $a'_{c} \cdot  c = a +\lambda^n_a(c) f_n$ with $\lambda^n_a(c) \in \Z$.
Consider:
\begin{equation*}
\begin{aligned}
A & := \frac{(a'_{c} \cdot  c)^{\varphi_n} - a^{\varphi_n}} {f_n \varphi_n} \\
&=\frac{ [a^{\varphi_n} + \lambda^n_a(c) f_n \varphi_n a^{\varphi_n-1} +
\lambda^n_a(c){}^2 f_n^2 \frac{\varphi_n(\varphi_n-1)}{2} 
a^{\varphi_n-2}+ \cdots] - a^{\varphi_n}}{f_n \varphi_n} \\
&\equiv \lambda^n_a(c) a^{\varphi_n-1} + \lambda^n_a(c){}^2 f_n 
\Frac{(\varphi_n-1)}{2} a^{\varphi_n-2}  \\
&\equiv  \lambda^n_a(c) a^{\varphi_n-1} 
\equiv \lambda^n_a(c) a^{-1} \!\!\! \pmod {p^{n+1}}, \\ 
\hbox{since}\  a^{\varphi_n}& \equiv 1 \pmod {qp^n}.
\end{aligned}
\end{equation*}

When $p=2$, one must take into account the term 
$\lambda^n_a(c) f_n \Frac{\varphi_n-1}{2} a^{\varphi_n-2} \sim
\Frac{1}{2}\lambda^n_a(c) f_n$, in which case the congruence is 
with the modulus $p^{n+1}$ (which is sufficient since for $n \geq n_0+e$,
this modulus annihilates ${\mathcal T}_K$ for any $p$).
\end{proof}

We have obtained for all $n \geq n_0+e$:
\begin{equation}\label{aln}
{\mathcal A}_{L_n}(c)  \ \wt{\equiv}\  \sm_{a=1}^{f_n}  \lambda^n_a(c) 
\cdot a^{-1}  \Big(\Frac{L_n}{a} \Big)\ \wt \equiv \ {\mathcal S}_{L_n}(c)^*,
\end{equation}
thus giving again, by restriction to $K$, the
annihilator ${\mathcal A}_{K,n}(c) \in \Z_p[G_K]$ of ${\mathcal T}_K$
such that (for all $n \geq n_0+e$)
${\mathcal A}_{K,n}(c)  \ \wt{\equiv}\  \sm_{a=1}^{f_n} \lambda^n_a(c) 
\, a^{-1} \Big(\Frac{K}{a} \Big)$.

\subsection{Normic properties of the ${\mathcal A}_{L_n}$ -- Euler factors}

\begin{theorem} \cite[Proposition II.2\,(iv)]{Gr6}.\label{thmfond}
Let $K$ be of conductor $f = m \ell$ where $m$ is the
conductor of a subfield $k$ of $K$ and where $\ell \ne p$ is a prime 
number. For $n \geq n_0$, let $L_n:= K(\mu_{qp^n})$ and 
the analogous field $l_n$ for $k$, of conductors $f_n$ and $m_n$,
repectively; recall that $\varphi_n = \varphi(q p^n)$. 

Let ${\mathcal A}_{L_n} := \Frac{-1}{f_n \varphi_n} 
\sm_a^{f_n} a^{\varphi_n} \Big(\Frac{L_n}{a} \Big)$
and ${\mathcal A}_{l_n} := \Frac{-1}{m_n \varphi_n} 
\sm_b^{m_n} b^{\varphi_n} \Big(\Frac{l_n}{b} \Big)$.
Then:
$${\rm N}_{L_n/l_n} ({\mathcal A}_{L_n}) 
 \ \wt{\equiv}\  \Big(1- \ell^{\varphi_n} \Frac{1}{\ell}
\Big(\Frac{l_n}{\ell} \Big)\Big){\mathcal A}_{l_n}, \ \ \hbox{resp.,}
\ \  {\rm N}_{L_n/l_n} ({\mathcal A}_{L_n}) 
 \ \wt{\equiv}\  {\mathcal A}_{l_n}, $$ 
if $\ell \nmid m$, resp., $\ell \mid m$
(congruences modulo $p^{n+1}\Z_p[G_n] + (1 - s_\infty)\Z_p[G_n]$).
\end{theorem}

\begin{proof} Suppose first that $\ell \nmid m$, so 
$f_n=l m_n$.\,\footnote{For $\ell=2$ and $m$ odd, $f=2m$ is not 
a conductor stricto sensu, but the following computations 
are exact and necessary with the modulus
$m_n$ and $f_n = 2\,m_n$; then if $f=2^k\cdot m$
($m$ odd, $k\geq 2$), the second case of the theorem applies and 
shall give the Euler factor 
$\big(1-2^{\varphi_n}\frac{1}{2} \big(\frac{l_n}{2}\big) \big)  \ \wt{\equiv}\ 
\big(1-\frac{1}{2} \big(\frac{l_n}{2}\big) \big)$.
If $p \mid f$ and $p \nmid m$, there is no Euler factor for $p$
since $m_n$ and $f_n$ are divisible by $p$; in other words, these
computations and the forthcoming ones are, by nature, not ``primitive'' at $p$.}
Put $a = b + \lambda\, m_n$, $\lambda  \in [0, \ell-1]$,
$b \in [1, m_n]$ prime to $m_n$;
since $a \in [1, f_n]$ is prime to $f_n$, $b$ is prime to $m_n$
and $\lambda  \ne \lambda_b^*$ such that $b + \lambda_b^* m_n 
=:  b'_\ell \cdot \ell$, $b'_\ell \in \Z$.
Thus $a^{\varphi_n}=( b + \lambda \, m_n)^{\varphi_n} 
\equiv  b^{\varphi_n} + b^{\varphi_n-1} \lambda\, m_n \varphi_n
\pmod {m_n \varphi_n p^{n+1}}$. Then:
\begin{equation*}
\begin{aligned}
{\rm N}_{L_n/l_n} ({\mathcal A}_{L_n}) & \ \wt{\equiv}\ 
\Frac{-1}{\ell m_n \varphi_n}\ \cdot  \sm_{b,\, \lambda \ne \lambda_b^*}
\Big[ b^{\varphi_n} + b^{\varphi_n-1} \lambda\, m_n \varphi_n \Big] \Big(\Frac{l_n}{b} \Big) \\
& \ \wt{\equiv}\  \Frac{-(\ell-1)}{\ell m_n \varphi_n}\sm_{b} b^{\varphi_n} \Big(\Frac{l_n}{b} \Big)
- \Frac{1}{\ell} \sm_{b,\, \lambda \ne \lambda_b^*} 
b^{\varphi_n-1} \lambda \Big(\Frac{l_n}{b} \Big) \\
&  \ \wt{\equiv}\  \Big(1-\Frac{1}{\ell}\Big) {\mathcal A}_{l_n}
- \Frac{1}{\ell}  \sm_{b,\, \lambda \ne \lambda_b^*} 
b^{\varphi_n-1} \lambda  \Big(\Frac{l_n}{b} \Big)\\
&  \ \wt{\equiv}\   \Big(1-\Frac{1}{\ell}\Big) {\mathcal A}_{l_n}
-\Frac{1}{\ell}  \sm_{b} b^{\varphi_n-1} \Big(\Frac{l_n}{b} \Big)
\Big( \sm_{\lambda \ne \lambda_b^*}  \lambda \Big)  \\
&  \ \wt{\equiv}\   \Big(1-\Frac{1}{\ell}\Big) {\mathcal A}_{l_n}
-\Frac{1}{\ell}  \sm_{b} b^{\varphi_n-1}  \Big(\Frac{l_n}{b} \Big)
\Big( \Frac{\ell (\ell-1)}{2} -  \lambda_b^* \Big).
\end{aligned}
\end{equation*}

We remark that $\lambda_b^* = \lambda_b^n(\ell)$ is relative to the writing
$b'_\ell \cdot \ell = b  + \lambda_b^n(\ell) \, m_n$ 
and that $b^{\varphi_n-1} \equiv b^{-1} \pmod {p^{n+1}}$, whence
using $\sum_{b} b^{-1} \big(\frac{l_n}{b} \big)  \ \wt{\equiv}\  0$:
$${\rm N}_{L_n/l_n} ({\mathcal A}_{L_n}) 
 \ \wt{\equiv}\  \Big(1-\Frac{1}{\ell}\Big) {\mathcal A}_{l_n}
 + \Frac{1}{\ell}  \sm_{b}\lambda_b^* \cdot b^{-1}  \Big(\Frac{l_n}{b} \Big).$$

But as we know (see relations \ref{defA} and \eqref{aln}),
$\sm_{b} \lambda_b^*\, b^{-1}  \Big(\Frac{l_n}{b} \Big)
 \ \wt{\equiv}\  {\mathcal A}_{l_n}(\ell)$; so
${\rm N}_{L_n/l_n}  ({\mathcal A}_{L_n}) 
 \ \wt{\equiv}\  \Big(1-\Frac{1}{\ell}\Big) {\mathcal A}_{l_n} 
 +  \Frac{1}{\ell}  {\mathcal A}_{l_n}(\ell)$:
since 
${\mathcal A}_{l_n}(\ell) \ \wt\equiv\  \Big(1- \ell^{\varphi_n} 
\Big(\Frac{l_n}{\ell} \Big)\Big)
{\mathcal A}_{l_n}$, we get
${\rm N}_{L_n/l_n}  ({\mathcal A}_{L_n}) 
\ \wt \equiv \  \Big(1- \ell^{\varphi_n} \Frac{1}{\ell}
 \Big(\Frac{l_n}{\ell} \Big) \Big){\mathcal A}_{l_n}$.

\medskip
The case $\ell \mid m$ is obtained more easily from
the same computations.
\end{proof}

Of course, for all $h \geq 0$ we get:
$${\rm N}_{L_{n+h}/L_n} ({\mathcal A}_{L_{n+h}})
\ \wt{\equiv}\  {\mathcal A}_{L_n}, $$ 
which expresses the coherence of the family $\big( {\mathcal A}_{L_n}\big)_n$
in the cyclotomic tower.

\begin{corollary} \label{coro} (i)
Let $K/k$ be an extension of fields of conductors 
$f_K$ and $f_k$, respectively.
Multiplying by $\Big[ 1 - c^{\varphi_n} \Big(\Frac{l_n}{c} \Big) \Big]
= {\rm N}_{L_n/l_n} \Big[1 - c^{\varphi_n} \Big(\Frac{L_n}{c} \Big) \Big]$
to get elements in the algebras $(\Z/p^{n+1}\Z)[{\rm Gal}(L_n/\Q)]$ 
and $(\Z/p^{n+1}\Z)[{\rm Gal}(l_n/\Q)]$, one obtains
$\ {\rm N}_{L_n/l_n} ({\mathcal A}_{L_n}(c))  \ \wt{\equiv}\ 
\prd_{\ell \mid f_K,\, \ell \nmid pf_k} 
\Big(1-\Frac{1}{\ell}  \Big(\Frac{l_n}{\ell} \Big) \Big){\mathcal A}_{l_n}(c)$.

\smallskip
(ii) Let ${\mathcal A}_{K,n}(c)$ and ${\mathcal A}_{k,n}(c)$ be the 
restrictions of ${\mathcal A}_{L_n}(c)$ and ${\mathcal A}_{l_n}(c)$
to $K$ and $k$, respectively; then
${\rm N}_{K/k} ({\mathcal A}_{K,n}(c))  \ \wt{\equiv}\ 
\prd_{\ell \mid f_K,\, \ell \nmid pf_k}\!\! \Big(1-\Frac{1}{\ell} \Big(\Frac{k}{\ell} \Big) \Big)
\cdot {\mathcal A}_{k,n}(c)$.

\smallskip
(iii) The family $({\mathcal A}_{K,n})_n = ({\rm N}_{L_n/K}({\mathcal A}_{L_n}))_n$ 
defines a pseudo-measure denoted ${\mathcal A}_K$ by abuse, such that 
the measure $({\mathcal A}_{K,n}(c))_n$ defines the element 
${\mathcal A}_{K}(c)=\Big (1 - \Big(\Frac{K}{c} \Big)\Big)
\cdot {\mathcal A}_{K} \in \Z_p[G_K]$ and gives the main formula:
$${\rm N}_{K/k} ({\mathcal A}_{K}(c)) \ \wt{\equiv}\ 
\prd_{\ell \mid f_K,\, \ell \nmid pf_k}\!\! \Big(1-\Frac{1}{\ell} \Big(\Frac{k}{\ell} \Big) \Big)
\cdot {\mathcal A}_{k}(c).$$
\end{corollary}

\begin{remarks}\label{rema4.10}
 (i) In a numerical point of view, we only need a minimal value of $n$, and
we shall write (e.g., for $n=e$ when $K \cap \Q_\infty = \Q$):
$${\mathcal A}_{K,e}(c) \ \wt\equiv \ 
\sum_{\sigma \in G_K} \!\!\! \Big[\ \sm_{a, \,(\frac{K}{a})=\sigma}
\lambda^e_a(c)\, a^{-1}\  \Big] \cdot \sigma =: 
\sm_{\sigma \in G_K} \Lambda^e_\sigma(c) \cdot \sigma.$$

Then the next step shall be to interprete the limit, $\Lambda_\sigma(c)$,
of the coefficients $\Lambda^n_\sigma(c) = \sum_{a,\, (\frac{K}{a})=\sigma}
\,\lambda^n_a(c)\, a^{-1}$, for $n \to \infty$, giving an equivalent annihilator, 
 but with a more canonical interpretation. 

\smallskip
(ii) In \cite{Gr5,Gr6,O,Sol1,Th,Ts,B-N,Ng2,Sol2,N-N, All1,All2,B-M}, 
some limits are expressed by means of $p$-adic logarithms of 
cyclotomic numbers/units of $\Q^f$ as expressions 
of the values at $s=1$ of the $p$-adic $L$-functions of~$K$
(for instance, in \cite[Theorem 2.1]{Ts} a link between 
Stickelberger elements and cyclotomic units is given following Iwasawa and
Coleman).
But these results are obtained with various non-comparable techniques;
this will be discussed later.

\smallskip
(iii) In the relation ${\mathcal A}_K(c) := 
\Big[1- \Big(\Frac{K}{c} \Big)\Big] {\mathcal A}_K$,
the choice of $c$ must be such that the integers $1- \chi(c)$ 
be of minimal $p$-adic valuation for the characters $\chi$ of $K$.
But $1-\chi(c)$ is invertible if and only if $\chi(c)$
is not a root of unity of $p$-power order.
\end{remarks}

\section{Remarks about Solomon's annihilators}
We shall give two examples: one giving the same annihilator as our's, and
another giving a Solomon annihilator in part degenerated, contrary to 
${\mathcal A}_K(c)$.

\subsection{Cubic field of conductor $1381$ and 
Solomon's $\Psi_K$}\label{1381}
We have (see the previous table of \S\,\ref{cubiccase}) 
$P=x^3 + x^2 - 460x - 1739$ and the classical program gives 
the class number in {\sf h}, the group structure of 
${\mathcal T}_K$ (in {\sf L}) and the units in {\sf E}:

\smallskip\footnotesize
\begin{verbatim}
{P=x^3+x^2-460*x-1739;K=bnfinit(P,1);p=7;nt=8;Kpn=bnrinit(K,p^nt);r=1;
Hpn=component(component(Kpn,5),2);C8=component(K,8);E=component(C8,5);
h=component(component(C8,1),1);L=List;i=component(matsize(Hpn),2);
for(k=1,i-1,c=component(Hpn,i-k+1);if(Mod(c,p)==0,
listinsert(L,p^valuation(c,p),1)));print(L);print("h=",h," ",L," E=",E)}

h=1    List([343, 7])
E=[245/13*x^2-4606/13*x-21522/13, 147/13*x^2+3479/13*x+11272/13]
\end{verbatim}

\smallskip\normalsize
So, the class group is trivial, ${\mathcal T}_K={\mathcal R}_K 
\simeq \Z/7^3\Z \times \Z/7\Z$
and the cyclotomic units are the fundamental units.
Then we shall use a definition of the automorphism $\sigma$
to define the Galois operation on the units:

\smallskip\footnotesize
\begin{verbatim}
{P=x^3 + x^2 - 460*x - 1739;print(nfgaloisconj(P))}
[x, -1/13*x^2 - 2/13*x + 302/13,  1/13*x^2 - 11/13*x - 315/13]
\end{verbatim}

\smallskip\normalsize
From $\varepsilon=\frac{245}{13} x^2 - \frac{4606}{13}x - \frac{21522}{13}$ 
and $\sigma : x \mapsto -\frac{1}{13}x^2 - \frac{2}{13}x + \frac{302}{13}$, one gets:

\footnotesize
\begin{verbatim}
Mod(245/13*(-1/13*x^2 - 2/13*x + 302/13)^2 - 
4606/13*(-1/13*x^2 - 2/13*x + 302/13) - 21522/13,P)=
Mod(147/13*x^2 + 3479/13*x + 11259/13, x^3 + x^2 - 460*x - 1739)
\end{verbatim}

\normalsize
which is $\varepsilon^\sigma$ and the units are, on the $\Q$-base
$\{1, x, x^2\}$:

\medskip
$\varepsilon = \varepsilon_1 =\frac{245}{13}x^2 - \frac{4606}{13}x - \frac{21522}{13}$,

$\varepsilon^\sigma= \varepsilon_2 =\frac{147}{13}x^2 + \frac{3479}{13}x + \frac{11259}{13}$,

$\varepsilon^{\sigma^2}= \varepsilon_3 = -\frac{392}{13}x^2 + \frac{1127}{13}x + \frac{175948}{13}$.

\smallskip
The second unit given by PARI is $\frac{147}{13}x^2 + \frac{3479}{13}x + \frac{11272}{13}
= - \varepsilon^{-\sigma^2}$. The order of $\varepsilon$ modulo $p=7$ is $114$.
We compute $A_i:= \varepsilon_i^{114}$ modulo $7^6$, $i=1,2,3$), then $L_i := A_i - 1$:

\footnotesize
\begin{verbatim}
{P=x^3+x^2-460*x-1739;
E1=Mod(245/13*x^2-4606/13*x-21522/13,P+Mod(0,7^6));
E2=Mod(147/13*x^2+3479/13*x+11259/13,P+Mod(0,7^6));
E3=Mod(-392/13*x^2+1127/13*x+175948/13,P+Mod(0,7^6));
L1=E1^114-1;L2=E2^114-1;L3=E3^114-1;
print(lift(L1)," ",lift(L2)," ",lift(L3))}
\end{verbatim}

\normalsize
$L_1=17542x^2+48608 x+81879=7^2 (358 x^2+992 x+1671)=7^2 \alpha_1$,

\smallskip
$L_2=62867 x^2+833 x+33761=7^2 (1283 x^2+17 x+689)=7^2 \alpha_2$,

\smallskip
$L_3=37240 x^2+68208 x+2009=7^2 (760 x^2+1392 x+41)=7^2 \alpha_3$,

\smallskip
giving $\Frac{1}{7} {\rm log} (\varepsilon_i) \equiv 7\alpha_i 
-\Frac{1}{2}7^3\alpha_i^2 \pmod {7^4}$:

\smallskip
$\Frac{1}{7} {\rm log} (\varepsilon) \equiv 791x^2 + 2142x + 378
=7(113x^2 + 306x + 54) \pmod{7^4}$,

$\Frac{1}{7} {\rm log} (\varepsilon^\sigma) \equiv 2121x^2 + 119x + 364
=7(303x^2 + 17x + 52) \pmod{7^4}$,

$\Frac{1}{7} {\rm log} (\varepsilon^{\sigma^2}) \equiv 1890x^2 + 140x + 1659
=7(270x^2 + 20x + 237) \pmod{7^4}$.

\smallskip
So, the Solomon annihilator $\Frac{1}{p} \sm_{\sigma \in G_K} 
{\rm log}(\varepsilon^\sigma) \cdot \sigma^{-1}$
of ${\mathcal T}_K$ is (modulo $7^3$ and
up to a $7$-adic unit):
$$\Psi_K \equiv 7 \cdot \big [15\,x^2 + 12\,x + 5 + (9\,x^2 + 17x + 3)\sigma^{-1}
+(25\,x^2 + 20\,x + 41) \sigma^{-2} \big ].$$

Since the norm is a trivial annihilator, we can replace $\Psi_K$ by
\begin{align*}
\Psi'_K & =\Psi_K - 7 \cdot(15\,x^2 + 12\,x + 5)(1+\sigma^{-1}+\sigma^{-2})  \\
& \equiv 7\cdot \big[ (43\,x^2 +5\,x +47)\, \sigma^{-1} +  
(10\,x^2 +8\,x + 36) \big ]\, \sigma^{-2} \pmod{7^3}. 
\end{align*}

Then, $43\,x^2 +5\,x +47$ is invertible $p$-adically (its norm is
prime to $7$) which gives the equivalent annihilator:
$$ 7\cdot \big[\sigma + (10\,x^2 +8\,x + 36)�\cdot (43\,x^2 +5\,x +47)^{-1} 
\equiv \sigma + 31 \pmod {7^2} \big]$$

equivalent to the annihilator defined by $ 7\cdot (\sigma - 18)$ modulo $7^3$.

\medskip
Our annihilator, given by the previous table,
is $1738+ 2186\,\sigma^{-1} +2361\,\sigma^{-2}$ equivalent to
$448 +623\, \sigma \equiv 7 \cdot (\sigma - 18) \pmod {7^3}$.

\smallskip
So $\sigma-18$ is an annihilator for the submodule 
${\mathcal T}_K^7 \simeq \Z/7^2 \Z$,
which is coherent since $18$ is of order 3 modulo $7^3$.

\smallskip
The perfect identity of the two results shows that no information has been lost 
for this particular case, whatever the method (but in the case of cyclic fields
of prime degree, there is not any Euler factor).

\subsection{Cyclic quartic field of conductor 
$37 \cdot 45161$ and Solomon's $\Psi_K$}\label{quartic}

Let $K$ be a real cyclic quartic field of conductor $f$ such that the quadratic 
subfield $k$ has conductor $m \mid f$, with for instance $f = \ell\, m$,
$\ell$ prime split in $k/\Q$. We take $p\equiv 1 \pmod 4$, $p \nmid f$.

\smallskip
Put $ \eta^{}_f:=1 - \zeta_f$,  $ \eta^{}_m:=1 - \zeta_m$,
$\eta^{}_K := {\rm N}_{\Q^f/K} (\eta^{}_f)$,
$\eta^{}_k := {\rm N}_{\Q^m/k} ( \eta^{}_m)$.

\smallskip
Then we have the Solomon annihilator:
$$\Psi_K = \Frac{1}{p} \sm_{\sigma \in G_K} 
{\rm log}( \eta_K^\sigma) \cdot \sigma^{-1}. $$

Since, from the formula \eqref{uc} (which applies since $m\ne 1$), one has
${\rm N}_{\Q^f/\Q^m} ( \eta^{}_f)=
 \eta_m^{(1-(\frac{\Q^m}{\ell})^{-1})}$, i.e., 
${\rm N}_{K/k} ( \eta^{}_K) = \eta_k^{(1-(\frac{k}{\ell})^{-1})} =1$,
we get (with $G_K = \{1, \sigma, \sigma^2, \sigma^3\}$):
\begin{equation*}
\begin{aligned}
\Psi_K & = \Frac{1}{p} \big( {\rm log}( \eta^{}_K)+ {\rm log}( \eta_K^\sigma) \cdot \sigma^{-1}
+ {\rm log}(\eta_K^{\sigma^2}) \cdot \sigma^{-2} + {\rm log}(\eta_K^{\sigma^3}) 
\cdot \sigma^{-3} \big) \\
& = \Frac{1}{p} \big( {\rm log}( \eta^{}_K)+ {\rm log}(\eta_K^\sigma) \cdot \sigma^{-1} 
- {\rm log}( \eta^{}_K) \cdot \sigma^{-2} - {\rm log}(\eta_K^{\sigma})
\cdot \sigma^{-3} \big) 
\end{aligned}
\end{equation*}

So, in this particular situation, one has:
\begin{equation}\label{psiK}
\Psi_K  = \Frac{1}{p} \big( {\rm log}( \eta^{}_K) +  
{\rm log}(\eta_K^\sigma)\cdot \sigma^{-1} \big) \cdot (1 - \sigma^2).
\end{equation}

Suppose that ${\mathcal T}_K$ is equal to the transfer of ${\mathcal T}_k$
(many examples are available), then ${\mathcal T}_K$ is annihilated by
$(1 - \sigma^2)$, whatever the structure of 
${\mathcal T}_k \simeq {\mathcal T}_K$; but one 
expects annihilators $A_K$ such that ${\rm N}_{K/k}(A_K)=A_k$
be a non-trivial annihilator of ${\mathcal T}_k$.

For instance, define $K$ by $x=\sqrt{\ell \sqrt m \Frac{\sqrt m + a}{2}}$
where $m=a^2+b^2$, $b=2 b'$. This gives the polynomial
$P=x^4 - \ell m \,x^2 +\ell^2 m b'{}^2$.
The following program gives many examples with non-trivial ${\mathcal T}_k$
(with $m$ prime, $p=5$):

\footnotesize\smallskip
\begin{verbatim}
{p=5;forprime(m=1,10^5,if(Mod(m,20)!=1,next);P=x^2-m;K=bnfinit(P,1);nt=12;
Kpn=bnrinit(K,p^nt);Hpn=component(component(Kpn,5),2);L=List;
i=component(matsize(Hpn),2);R=0;for(k=1,i-1,c=component(Hpn,i-k+1);
if(Mod(c,p)==0,R=R+1;listinsert(L,p^valuation(c,p),1)));if(R>0,
print("m=",m," structure",L)))}
\end{verbatim}

\normalsize \smallskip
For $m= 45161$, one obtains ${\mathcal T}_k \simeq \Z/5^5 \Z$; then
$a=205$, $b'=28$. Now we find some primes $\ell$ with the following program:

\footnotesize \smallskip
\begin{verbatim}
{p=5;m=45161;bprim=28;forprime(ell=7,10^3,if(Mod(ell,4)!=1,next);
if(kronecker(m,ell)!=1,next);P=x^4-ell*m*x^2+ell^2*m*bprim^2;
K=bnfinit(P,1);nt=12;Kpn=bnrinit(K,p^nt);Hpn=component(component(Kpn,5),2);
L=List;i=component(matsize(Hpn),2);
for(k=1,i-1,c=component(Hpn,i-k+1);if(Mod(c,p)==0,
listinsert(L,p^valuation(c,p),1)));
print("ell=",ell," m=",m," P=",P," structure",L))}
\end{verbatim}

\normalsize \smallskip
giving the following examples (for which ${\mathcal T}_k$ is 
a direct factor in ${\mathcal T}_K$):

\footnotesize 
\begin{verbatim}
ell=13  P=x^4-587093*x^2+5983651856       structure [3125]
ell=17  P=x^4-767737*x^2+10232398736      structure [3125,5,5]
ell=37  P=x^4-1670957*x^2+48471120656     structure [3125]
ell=997 P=x^4-45025517*x^2+35194105312016 structure [3125,25]
\end{verbatim}

\normalsize
We consider the case $\ell=37$, $P=x^4-1670957 x^2+48471120656$
for wich PARI gives the following information that may be used by the reader:

\footnotesize
\begin{verbatim}
nfgaloisconj(x^4-1670957*x^2+48471120656)=
[-x, x, -1/212380*x^3 + 43593/5740*x, 1/212380*x^3 - 43593/5740*x]

{P=x^4-1670957*x^2+48471120656;K=bnfinit(P,1);p=5;nt=8;Kpn=bnrinit(K,p^nt);
r=1; Hpn=component(component(Kpn,5),2);C8=component(K,8);E=component(C8,5);
h=component(component(C8,1),1);L=List;i=component(matsize(Hpn),2);R=0;
for(k=1,i-1,c=component(Hpn,i-k+1);if(Mod(c,p)==0,R=R+1; 
listinsert(L,p^valuation(c,p),1)));print("h=",h," ",L);print("E=",E)}

h=2   List([3125])
\end{verbatim}
\normalsize

Now, consider the annihilator ${\mathcal A}_{K,n}(c) =: A_K$;
since ${\mathcal T}_K \simeq {\mathcal T}_k$, we get 
${\mathcal T}_K^{A_K} \simeq  
{\mathcal T}_k^{{\rm N}_{K/k}(A_K)}$, where (see
Corollary \ref{coro}):
$${\rm N}_{K/k}({\mathcal A}_{K, n}(c)) \ \wt\equiv\ 
\Big(1-\Frac{1}{\ell}  \Big(\Frac{k}{\ell} \Big) \Big){\mathcal A}_{k,n}(c).$$ 

Then $\ell=37 \equiv 2 \pmod 5$ splits in $k$ and 
$1-\Frac{1}{\ell}  \Big(\Frac{k}{\ell} \Big) =
1-\Frac{1}{\ell}$ is invertible modulo $5$.

\smallskip
So $A_K$ acts on ${\mathcal T}_K$ as  
${\mathcal A}_{k,n}(c)$ on ${\mathcal T}_k$; we can use
the program for quadratic fields and $p>2$ (of course
the bounds $bf, Bf$ may be arbitrary):

\footnotesize
\begin{verbatim}
{p=5;nt=8;bf=45161;Bf=45161;for(f=bf,Bf,v=valuation(f,2);M=f/2^v;
if(core(M)!=M,next);if((v==1||v>3)||(v==0 & Mod(M,4)!=1)||
(v==2 & Mod(M,4)==1),next);P=x^2-f;K=bnfinit(P,1);Kpn=bnrinit(K,p^nt);
C5=component(Kpn,5);Hpn0=component(C5,1);Hpn=component(C5,2);
h=component(component(component(K,8),1),2);L=List;ex=0;
i=component(matsize(Hpn),2);for(k=1,i-1,co=component(Hpn,i-k+1);
if(Mod(co,p)==0,val=valuation(co,p);if(val>ex,ex=val);
listinsert(L,p^val,1)));Hpn1=component(Hpn,1);
vptor=valuation(Hpn0/Hpn1,p);tor=p^vptor;S0=0;S1=0;pN=p*p^ex;fn=pN*f;
for(cc=2,10^2,if(gcd(cc,p*f)!=1 || kronecker(f,cc)!=-1,next);c=cc;break);
for(a=1,fn/2,if(gcd(a,fn)!=1,next);asurc=lift(a*Mod(c,fn)^-1);
lambda=(asurc*c-a)/fn;u=Mod(lambda*a^-1,pN);
s=kronecker(f,a);if(s==1,S0=S0+u);if(s==-1,S1=S1+u));
L0=lift(S0);L1=lift(S1);A=L1-L0;if(A!=0,A=p^valuation(A,p));
print(f," P=",P," ",L0," ",L1," A=",A," tor=",tor," T_K=",L," Cl_K=",h))}
\end{verbatim}

\normalsize
giving the annihilator $A_k \equiv 10185 + 3935\, \ov\sigma \pmod{5^6}$ 
where $\ov\sigma$ generates ${\rm Gal}(k/\Q)$; then,
$A_k$ is equivalent, modulo the norm, to the integer 
$10185 - 3935 \equiv 2 \cdot 5^5 \pmod{5^6}$, which is perfect since 
${\mathcal T}_k \simeq \Z/5^5 \Z$. 

\smallskip
The class group of $k$ being trivial, the fundamental unit 
$\varepsilon$ is the cyclotomic one and is such that
$\varepsilon^4=1+ 5^6 \cdot \alpha$, $\alpha$ prime to $5$,
which confirms that:
\begin{equation}\label{psik}
\Psi_k \sim \Frac{1}{5} ( {\rm log}(\varepsilon) + 
{\rm log}(\varepsilon^{\ov\sigma}) \cdot \ov\sigma) = 
\Frac{1}{5}\, {\rm log}(\varepsilon) (1-\ov\sigma)
\end{equation}

equivalent (modulo the norm) to $\frac{2}{5} \,{\rm log}(\varepsilon)$ 
and $\Psi_k = A_k$ as expected.
Meanwhile, the Solomon annihilator $\Psi_K$ does not give $\Psi_k$
by restriction, but~$0$.

\section{About the annihilator ${\mathcal A}_{K}(c)$
and the primitive $L_p(1,\chi)$}\label{section9}

\subsection{Galois characters v.s. Dirichlet characters}
Let $f_K$ be the conductor of $K$.
In most formulas, the characters $\chi$ of $K$ must be primitive of conductor 
$f_\chi \mid f_K$, whence Dirichlet characters on $(\Z/f_\chi \Z)^\times$
such that $\chi\big[\big(\frac{\Q^{f_\chi}}{a}\big)\big]$ makes sense for
$a \in \Z$, prime to $f_\chi$, but not necessarily for
$\chi\big[\big(\frac{\Q^{f_K}}{a}\big)\big]$ if a prime $\ell$ divides 
both $a$ and $f_K$ but not $f_\chi$. This is an obstruction to 
consider them as Galois characters 
over $\Z_p[G_K]$ for instance, whence defined on $(\Z/f_K \Z)^\times$; 
so we shall introduce the corresponding 
Galois character of $G_K$, denoted $\psi_\chi =: \psi$.
A Galois character $\psi$ of $G_K$ is also a character of 
$G_n = {\rm Gal}(L_n/\Q)$ whose kernel fixes $K$,
so $\psi(a)$ ($a \in [1, f_n]$ prime to $f_n$) is the image by $\psi$
of the Artin symbol $\Big(\Frac{L_n}{a} \Big)$ whence of
$\Big( \Frac{K}{a} \Big)$. 

\smallskip
Any non-primitive writing $\psi({\mathcal A}_K)$, 
for ${\mathcal A}_K \in \Z_p[G_K]$, may introduce a product 
of Euler factors. Indeed, let $k_\chi$ be the subfield fixed 
by the kernel of $\psi=\psi_\chi$ (then $\chi$ is a primitive character of $k_\chi$ 
but not necessarily of $K$); then,
$\psi({\mathcal A}_K)= \psi ({\rm N}_{K/k_\chi}({\mathcal A}_K))
= \chi( {\mathcal E}_{k_\chi})\cdot \chi({\mathcal A}_{k_\chi})$
in which $\chi( {\mathcal E}_{k_\chi})$ may be non-invertible (or $0$).

\subsection{Expression of $\psi({\mathcal A}_{K}(c))$}
Let $\psi$ be any Galois character of $K$ considered as Galois 
character of ${\rm Gal}(L_n/\Q)$, for $n \geq n_0+e$.
We then have the following result about the computation of the
annihilator ${\mathcal A}_{K}(c)=: 
\sm_{\sigma \in G_K} \Lambda_\sigma(c) \cdot \sigma$ 
(given explicitely by the Theorem \ref{st*}), without any hypothesis
on $K$ and $p$:

\begin{lemma}\label{lambda}
The expression $\psi({\mathcal A}_{K}(c))$ is the product of 
the multiplicator $1 - \psi \big(\big(\frac{L_\infty}{c} \big) \big)$ by
the non-primitive value $L_p(1, \psi)$. In other words, one has:
\begin{equation*}
\begin{aligned}
\psi({\mathcal A}_{K}(c)) & =  (1-\psi (c)) \cdot L_p(1,\psi) \\
&= (1-\psi (c))  \cdot  \!\! \prd_{\ell \mid f_K,\, \ell \nmid p f_\chi}
 \!\!\big(1-\chi(\ell) \ell^{-1}\big) \,L_p(1,\chi) .
\end{aligned}
\end{equation*}
\end{lemma}

\begin{proof}
This comes from the classical construction of $p$-adic $L$-functions
\cite[Propositions II.2, II.3, D\'efinition II.3, II.4, and Remarques II.3, II.4]{Gr6},
then \cite[page 292]{F}. Thus we obtain, using the computations of the
\S\,\ref{defAc}, the link between the limit (for $n \to \infty$):
$$\psi({\mathcal A}_{K}(c)) = \sm_{\sigma \in G_K} 
\!\Lambda_ \sigma(c) \cdot \psi(\sigma)\ \hbox{ (cf. Remark \ref{rema4.10}\,(i)),} $$ 
of $\psi({\mathcal A}_{L_n}(c))= \psi({\mathcal A}_{K,n}(c)) =
\sm_{\sigma \in G_K} \Lambda^n_ \sigma(c) \, \psi(\sigma)$,
and the value at $s=1$ of the $L_p$-function of the {\it primitive character}
$\chi$ associated to $\psi$.
\end{proof}

\begin{remark}
Note that in the various calculations in \S\,\ref{defAc}, 
$\varphi_n=\varphi(qp^n)$ when $n \to \infty$ plays the role of $1-s$ when 
$s \to 1$ in the construction of $p$-adic 
$L_p$-functions by reference to Bernoulli numbers.
\end{remark}

For all primitive Dirichlet character $\chi \ne 1$ of $K$, of modulus $f_\chi$ 
(or $p f_\chi$ if $p\nmid f_\chi$), and for all $p\geq 2$, we have
the classical formulas of the value at $s=1$ of the $p$-adic 
$L$-functions (see for instance \cite[Theorem 5.18]{W}), where 
${\tau}(\chi) = \sum_{(a, f_\chi)=1} \chi(a) \zeta_{f_\chi}^a$ 
is the primitive Gauss sum of $\chi$:
\begin{equation*}
L_p(1,\chi) = -\Big(1-\Frac{\chi(p)}{p} \Big) \cdot \Frac{\tau(\chi)}{f_\chi}
\sm_{a\in[1, f_\chi],\,(a , f_\chi)=1} \chi^{-1}(a) {\rm log}(1 - \zeta_{f_\chi}^a), 
\end{equation*}

where the Euler factor $1- \chi(p) p^{-1}$ illustrates the fact that for 
$L_p$-functions, any character $\chi$ is considered modulo $p f_\chi$
when $p \nmid f_\chi$.

\smallskip
From the Coates formula \cite{C} and classical computations (see also some
details in \cite[\S\,2.2]{Gr4}) we recall that 
$\order {\mathcal T}_K \sim  [K \cap \Q_\infty : \Q] \cdot \prod_{\chi \ne 1}
\hbox{$\frac{1}{2}$}\,L_p (1,\chi)$ (up to a $p$-adic unit), thus
$\order {\mathcal T}_K \sim \prod_{\chi \ne 1} \hbox{$\frac{1}{2}$}\,L_p (1,\chi)$
if $K \cap \Q_\infty = \Q$ (i.e., $n_0=0$).
Moreover, we know that in the semi-simple case, one obtains the orders 
of the isotypic components of ${\mathcal T}_K$ by means of the 
$\frac{1}{2}\,L_p (1,\chi)$; but the whole Galois structure of ${\mathcal T}_K$
is more precise that the set of those given by the components ${\mathcal T}_K^{e_\theta}$,
where the $e_\theta$ are the corresponding $p$-adic idempotents.

\begin{remark} \label{solo}
Let $\chi$ be a primitive Dirichlet character of conductor $f_\chi \ne 1$.
We define the ``modified Solomon element'' of $\Z_p[G_{k_\chi}]$:
\begin{equation*}
 \Psi_{k_\chi} := -\Big(1- \Frac{\chi(p)}{p}\Big)
\cdot   \Frac{\tau(\chi)}{f_\chi} 
\sm_{\tau \in G_{k_\chi}} {\rm log}(\eta_{k_\chi}^\tau)\cdot \tau^{-1}.
\end{equation*}

Whence $L_p(1,\chi) =  \chi\big( \Psi_{k_\chi})$ ($\chi \ne 1$ primitive). Put:
$$C_\chi := -\Big(1- \Frac{\chi(p)}{p}\Big) \cdot \Frac{\tau(\chi)}{f_\chi}.$$
When $p\nmid {f_\chi}$, $\tau(\chi)$ is 
invertible and $C_\chi \cdot {\rm log}(\eta_{k_\chi}^\tau)
\sim \frac{1}{p} \cdot {\rm log}(\eta_{k_\chi}^\tau) \sim \Psi_{k_\chi}$
(the original Solomon element); when $p \mid f_\chi$, the factor 
$\frac{1}{p}$ in $C_\chi$ is replaced, ahead the logarithms, by the 
quotient $\Frac{1}{\overline {\tau(\chi)}}$ having the suitable 
$p$-valuations. For instance, if $d$ is prime and $p$ unramified,
$\frac{1}{p} \sm_{\sigma \in G_K} {\rm log}(\eta_{K}^ \sigma)\cdot \sigma ^{-1}$
annihilates ${\mathcal T}_K$.
\end{remark}

\subsection{The annihilator ${\mathcal A}_{K}(c)$ and the $ \Psi_{k_\chi}$}
The following statement does not assume any hypothesis on $K$ and $p$
and gives again the known results of annihilation (e.g., semi-simple case,
but also the point of view of \cite{O}):

\begin{theorem} \label{thmp}
Let $K$ be a real abelian number field, of degree $d$, of Galois group $G_K$
and of conductor $f_K$.
Let ${\mathcal A}_{K}(c) =\ds \lim_{n\to\infty}{\mathcal A}_{K,n}(c) \in \Z_p[G_K]$
annihilating ${\mathcal T}_K$ (cf. Theorem \ref{st*}).
Then we have (where each $\chi$ is the primitive 
Dirichlet character associated to the Galois character $\psi$ of $G_K$):\par

\medskip
\centerline{$\ds{\mathcal A}_{K}(c) =  \Frac{1}{d} \sum_{\sigma \in G_K} 
\!\!\Big[ \sm_{\psi \ne 1} \psi^{-1} (\sigma) (1 - \psi (c)) \ \cdot  \!\!\!\!
\prd_{\ell \mid f_K,\, \ell \nmid pf_\chi}\!\!\! \Big(1-\Frac{\chi(\ell)}{\ell}\Big) \cdot
\chi( \Psi_{k_\chi}) \Big] \cdot \sigma$,}\par
with $ \Psi_{k_\chi} = - \Big(1- \Frac{\chi(p)}{p}\Big)\, \Frac{\tau(\chi)}{f_\chi}
\sm_{\tau \in G_{k_\chi}}  {\rm log} \big({\rm N}_{\Q^{f_\chi}/ k_\chi} 
(1 - \zeta_{f_\chi})^\tau \big)\cdot \tau^{-1}$.

\smallskip
Thus, ${\mathcal T}_K$ is annihilated by the ideal 
${\mathfrak A}_{K}$ of $\Z_p[G_K]$ generated by the
${\mathcal A}_{K}(c)$, $c \in \Z$, prime to $2\,p f_K$.
\end{theorem}

\begin{proof} For all Galois character $\psi$ of $G_K$,
Lemma \ref{lambda} leads to the identity:
\begin{equation*}
\begin{aligned}
\psi ({\mathcal A}_{K}(c)) & = \sm_{\sigma \in G_K} \Lambda_\sigma(c)
\cdot \psi (\sigma) \\
& = (1 - \psi (c)) \cdot \prd_{\ell \mid f_K,\, \ell \nmid pf_\chi}\!\!\! \big(1-\chi(\ell) \ell^{-1}\big)
\cdot L_p(1,\chi)  \\
& =  (1 - \psi (c)) \cdot  \prd_{\ell \mid f_K,\, \ell \nmid pf_\chi}\!\!\! \big(1-\chi(\ell) \ell^{-1}\big)
\cdot \chi( \Psi_{k_\chi})
\end{aligned}
\end{equation*}
with $\psi_1 ({\mathcal A}_{K}(c)) = 0$ for the unit character $\psi_1$.

\smallskip
Since the matrix $\big(\psi (\sigma)\big)_{\psi, \sigma}$ is invertible
with inverse $\Frac{1}{d}\Big(\psi^{-1} (\sigma)\Big)_{\sigma, \psi}$, this yields
$\Lambda_\sigma(c) = \Frac{1}{d} \sm_ \psi \psi^{-1} (\sigma) 
\psi({\mathcal A}_{K}(c))=
\Frac{1}{d} \sm_ \psi \psi^{-1} (\sigma)(1 - \psi (c)) \cdot L_p(1, \psi)$. 
Whence the result using the expression of $L_p(1, \psi)$ in Lemma \ref{lambda}.
\end{proof}

\subsection{A cyclic quartic field $K$ of conductor $37 \cdot 45161$}

We recall from \S\,\ref{quartic} that $m= 45161$ is totally ramified in $K$, that
$\ell = 37$ splits in the quadratic subfield $k=\Q(\sqrt m)$ and is ramified in $K/k$;
then $p=5$ totally splits in $K$. We have ${\mathcal T}_k \simeq \Z/5^5\Z$.

\smallskip
Denote the four characters by $\psi_1$, $\psi_2$, $\psi_4\  \&\  \psi_4^{-1}$
(orders $1$,$2$, $4$, respectively) and let 
$G_K = \{1, \sigma^2, \sigma, \sigma^{-1}\}$ with $\sigma$ of order $4$.
We shall put $\psi_4(\sigma)=i$, and so on by conjugation and the relation
$\psi_2 = \psi_4^2$. 

\smallskip
Then, using the modified Solomon elements $\Psi_k$, $\Psi_K$
(expressions \eqref{psiK}, \eqref{psik}):
$$ \Psi_k = 5^5  \!\cdot\! u\ \ \  \& \ \ \ 
\Psi_K = \Frac{v}{5} \big({\rm log}(A) + 
{\rm log}(B)\, \sigma \big) (1 - \sigma^2),$$

where $u$ and $v$ are $p$-adic units, $A \ \&\  B=A^\sigma$ are 
the two independent units of $K$ of relative norm 1. 

\smallskip
We have to compute the coefficients $\psi^{-1} (\sigma) (1 - \psi (c))$, 
which gives the array:
\footnotesize
\begin{align*}
&\hspace{0.55cm}& \psi_1\hspace{1.35cm}&\ \ \ \ \psi_2&
     \psi_4 \hspace{1.9cm}&\ \ \ \ \ \psi_4^{-1} \\
&1  & 0 \hspace{1.5cm} &\ \ \ 1\cdot 2   &1\cdot (1-i) \hspace{1.5cm}  & \ \ \  1\cdot(1+i)   \\
&\sigma^2  & 0 \hspace{1.5cm} &\ \ \ 1\cdot 2  & -1 \cdot(1-i) \hspace{1.4cm} & -1 \cdot(1+i)  \\
&\sigma     &  0 \hspace{1.5cm}  &  -1\cdot 2  &  -i \cdot (1-i)\hspace{1.5cm} & \ \ \ \ i \cdot(1+i)  \\
&\sigma^{-1} & 0 \hspace{1.5cm} &  -1\cdot 2   &  i \cdot (1-i) \hspace{1.5cm} &\ - i \cdot(1+i)
\end{align*}

\smallskip\normalsize
Then the terms $\prd_{\ell \mid f_K,\, \ell \nmid pf_\chi}\!\!\!\big(1-\chi(\ell) \ell^{-1}\big) \cdot
\chi( \Psi_{k_\chi})$ have the following values, depending on the character
$\psi$  in the summation of the theorem: 

\smallskip
$\bullet\  \ 5^5 \!\cdot\! u$ for $\psi_2$, since
$1 - \chi_2(\ell) \ell^{-1} = 1 - 37^{-1} \sim 1$,

\smallskip
$\bullet\ \  \Frac{2 v}{5} \big( {\rm log}(A) +i \,{\rm log}(B) \big)\ \ \&\ \ 
 \Frac{2 v}{5} \big( {\rm log}(A) - i \,{\rm log}(B) \big)$,
for $\psi_4 \ \ \&\ \  \psi_4^{-1}$.

\medskip
We obtain, up to a $p$-adic unit, using the coefficients of the above array:
\begin{equation*}
\begin{aligned}
&{\mathcal A}_{K}(c) = \\
& \Big[ \Frac{v}{5}\big[ {\rm log}(A) + {\rm log}(B) \big] + 5^5 \!\cdot\! u \Big]
 + \Big[ \Frac{v}{5} \big[- {\rm log}(A) - {\rm log}(B) \big] + 5^5  \!\cdot\! u \Big] \cdot\sigma^2 + \\
& \Big[ \Frac{v}{5} \big[- {\rm log}(A) + {\rm log}(B) \big] - 5^5  \!\cdot\! u\Big] \!\cdot \! \sigma 
+ \Big[ \Frac{v}{5} \big[{\rm log}(A) - {\rm log}(B) \big] - 5^5  \!\cdot\! u\Big] \cdot \sigma^{-1}  \\
& = 5^5 u\! \cdot\! (1 - \sigma) (1 + \sigma^2)  \\
&  \hspace{2.2cm} + v \Big [\Frac{1}{5} \big[ {\rm log}(A) + {\rm log}(B)\big]
- \Frac{1}{5} \big[ {\rm log}(A) - {\rm log}(B) \big] \cdot \sigma\Big] \cdot (1-\sigma^ 2).
\end{aligned}
\end{equation*}

We give $A$, one of the two units of relative norm $1$ (the other
being $B=A^\sigma$):

\smallskip\footnotesize
\begin{verbatim}
377216797578975495402206020260112295002483855252847326395960961891321756
935656033880097414072613343385538964199960251752277854265043908282068622
071287/424760*x^3 - 
863005972214749996449837366815586234260744443520807110375190268414267539
937539821074892103868728835668111842347981799323725052575447796376125480
7708541/7585*x^2 - 
301058401703043815651487372068244675606729686675124486738439428208587682
003249385550605088262234049232685807258542997079887400411162925713036023
300228411/11480*x + 
137753779960320144069066397981124894126287808388246384703621136571725449
454295610577594731673630502306081901547245942649393930683936045056394190
29007385081/410
\end{verbatim}

\normalsize
So it is easy to compute $A^4-1$, congruent modulo $5^8$ to:
$$5 \cdot \alpha = 317056\,x^3+260605\,x^2+260934\,x+182595, $$
whence ${\rm log}(A) \sim 5 \cdot \alpha$.
The decompositions into prime ideals of $5$ (which is totally split in $K/\Q$)
and of $5 \cdot \alpha$ give respectively for the $5$-places:

\smallskip\footnotesize
\begin{verbatim}
[[5,[-3,-2,2,2]~,1,1,[3,4,1,1]~]1]    [[5,[-3,0,2,-2]~,1,1,[2,0,4,1]~]1]
[[5,[-1,-2,-2,-2]~,1,1,[1,1,1,1]~]1]  [[5,[0,-1,-2,2]~,1,1,[2,2,4,1]~]1]

[[5,[-3,-2,2,2]~,1,1,[3,4,1,1]~]2]    [[5,[-3,0,2,-2]~,1,1,[2,0,4,1]~]1]
[[5,[-1,-2,-2,-2]~,1,1,[1,1,1,1]~]2]  [[5,[0,-1,-2,2]~,1,1,[2,2,4,1]~]1]
\end{verbatim}

\normalsize
Dividing by $5$, we find that
$\Frac{1}{5}{\rm log}(A) \sim \pi_1\cdot \pi_2$ then 
$\Frac{1}{5}{\rm log}(A^\sigma) \sim (\pi_1\cdot \pi_2)^\sigma =:
\pi_3\cdot \pi_4$, where the $\pi_i$ are integers with valuation $1$ at 
the four prime ideals dividing $5$; thus the coefficient:
\begin{equation*}
\begin{aligned}
U-V\,\sigma & = \Frac{1}{5}{\rm log}(A\,B) - \Frac{1}{5}{\rm log}(A\,B^{-1}) \\
& \sim u\,\pi_1\cdot \pi_2 + u'\, \pi_3\cdot \pi_4 
-(u\,\pi_1\cdot \pi_2 - u'\, \pi_3\cdot \pi_4) \cdot \sigma,
\end{aligned}
\end{equation*}
of $1-\sigma^ 2$ in ${\mathcal A}_{K}(c)$ is such that:
$$U^2+V^2 \equiv 2\,(u^2\,\pi_1^2\cdot \pi_2^2 + 
u'{}^2\,\pi_3^2\cdot \pi_4^2) \pmod 5$$ 
is $5$-adically invertible. So 
${\mathcal A}_{K}(c) = 5^5 u (1-\sigma) (1+\sigma^2) + w (1-\sigma^2 )$,
$u, w$ invertible.
This gives the optimal annihilation of 
both ${\mathcal T}_{k}$ (since ${\mathcal T}_K =j_{K/k}({\mathcal T}_k)$),
and the relative factor ${\mathcal T}_{K}^{\,*} = 1$, as kernel of the 
relative norm $1+\sigma^2$ in $K/k$, since the operation is given by 
$U-V\,\sigma$  which is invertible.

\subsection{A cyclic quartic field $K$ of conductor $2^2 \cdot 16212 \cdot 677$}
Let $K = \Q(x)$ where $x=\sqrt{677\,\Frac{1621+39 \sqrt{1621}}{2}}$.
This field is also defined by $P=x^4 - 1097417x^2 + 18573782725$.
The conjugates of $x$ are given by:

\smallskip\footnotesize
\begin{verbatim}
 nfgaloisconj(P)=[-x, x, -1/132015*x^3+1571/195*x, 1/132015*x^3-1571/195*x]
\end{verbatim}
\normalsize

\smallskip
We still consider the case $p=5$.
The prime $\ell=677$ splits in the quadratic subfield 
$k=\Q(\sqrt{1621})$, the ramified prime $2$ does not split in $k$; 
the class number of $k$ is 1 and that of $K$ is 4, so
we obtain a trivial $5$-class group and the following group structures
giving, here, a non-trivial relative ${\mathcal T}_{K}^{\,*}$:
$${\mathcal T}_k \simeq \Z/ 5^2 \Z, \ \ {\mathcal T}_K 
\simeq \Z/ 5^2 \Z \times \Z/ 5^3 \Z. $$

In $k$, the cyclotomic unit is the fundamental unit and is given by:
$$\varepsilon = \Frac{119806883557}{26403} \,x^2 - \Frac{3042847629386}{39}; $$

we compute that $\frac{1}{5} \cdot{\rm log}(\varepsilon) \sim 5^2 \sim \Psi_k$ 
as expected since ${\mathcal T}_k={\mathcal R}_k$.

\smallskip
The cyclotomic units $A$ and $B=A^\sigma$ of $K$, 
of relative norm $1$, are too large to be given here, but we can work 
with some representatives modulo a large power of $5$. 
As in the previous example, we have to compute (up to $5$-adic units 
since the Euler factors for $2$ and $677$ are invertible):
\begin{equation}\label{677}
\Big [  \Frac{1}{5} \big [ {\rm log}(A) + {\rm log}(B) \big] 
-  \Frac{1}{5}  \big[ {\rm log}(A) - {\rm log}(B) \big] 
\cdot \sigma\Big] \cdot (1-\sigma^ 2).
\end{equation}

We see that ${\rm log}(A)$ is of the form $5 \cdot \alpha$, where 
$\alpha$ is a $5$-adic unit, and that
$\Frac{1}{5} \big[ {\rm log}(A) - {\rm log}(B) \big]$
and $\Frac{1}{5} \big [{\rm log}(A) + {\rm log}(B) \big]$
are $5$-adically invertible, so we consider for instance:
$$C := \Frac{ {\rm log}(A) + {\rm log}(B)}
{{\rm log}(A) - {\rm log}(B)} \equiv
13\cdot 5^2 x^3+5^3 x^2+19 \cdot 5^2 x+57 \pmod{5^4}$$ 

and we verify that, despite the denominators $5$,
$\Frac{3}{5} \cdot x^3 - \Frac{1}{5}  \cdot x$
is an integer of $K$ (congruent to $x^\sigma$ modulo $5$
as given by $nfgaloisconj(P)$) so that:
$$C \equiv 5^3\cdot 3 \cdot \Big(\Frac{3}{5} x^3 - \Frac{1}{5}  x + x^2\Big)+57
\pmod{5^4} .$$

Since the exponent of ${\mathcal T}_K$ is $5^3$,
we obtain that the coefficient $U - V \cdot \sigma$ (in \eqref{677}) is
equal to $(57 - \sigma) \cdot (1-\sigma^ 2)$; thus the whole annihilator is:
$${\mathcal A}_K(c) \equiv 5^2 \cdot u \cdot (1 - \sigma) (1 + \sigma^2) + 
v \cdot (57 - \sigma) \cdot (1-\sigma^ 2) \pmod{5^4}.$$

So, on the factor ${\mathcal T}_k$ the annihilator ${\mathcal A}_K(c)$
acts as the order $5^2$ of ${\mathcal T}_k$, and on the relative
submodule ${\mathcal T}_K^*$, it acts as $57 - \sigma$,
which is very satisfactory since $57$ is of order $4$ modulo $5^3$
(note that $57^2+1 = 5^3 \cdot 26$).

\medskip
These examples show that ${\mathcal A}_K(c)$ takes into account the
whole structure of ${\mathcal T}_K$; but when the Euler factor is not
a $p$-adic unit because of a prime $\ell \equiv 1 \pmod p$ which splits 
in $k$ and is ramified in $K/k$, the annihilation is probably not optimal. 

\smallskip
It should be usefull to know if the annihilators, given more recently in
the literature, have best properties or not in this point of view, which
is not easy since numerical tests are absent (to our knowledge).

\subsection{Ideal of annihilation for arbitrary real abelian number fields}
We do not make any assumption on $p$ and $G_K$, nor on the decomposition 
of the primes $\ell \mid f_K$ in the real abelian extension $K/\Q$.

\smallskip
 If $K/\Q$ is cyclic, one can choose $c$ (prime to $2 p f_K$) such that
for all $\psi \ne 1$, $1 - \psi (c)$ is non-zero with minimal $p$-adic valuation;
this valuation is $0$ as soon as $d$ is not divisible by $p$, taking 
$\big(\frac{K}{c}\big)$ as a generator of $G_K$. 
Since in the non-cyclic case, this is impossible, 
we can consider the augmentation ideal 
${\mathcal I}_K = \big\langle 1- \big(\frac{K}{c} \big), \ 
\hbox{$c$ prime to $2 p f_K$} \big \rangle_{\Z[G_K]}$ of $G_K$
and the ideal:
$${\mathcal I}_K \cdot {\mathcal A}_K$$

which annihilates ${\mathcal T}_K$. It is clear, from 
Corollary \ref{coro}, that the pseudo-measure ${\mathcal A}_K$ 
does not depend on ${\mathcal I}_K$ and that any choice of 
$\delta_K \in {\mathcal I}_K$ is such that
$\delta_K \,{\mathcal A}_K \in \Z_p[G_K]$. 

In a $p$-group $G_K$ of $p$-rank $r$,
$\delta_K = \sum_{i=1}^r \lambda_i \cdot(1- \sigma_i)$, where the
generators $\sigma_i$ are suitable Artin symbols of integers $c_i$ 
prime to $2 p f_K$; then the characters $\psi$ may be written
$\psi = \prod_{i=1}^r \psi_i$, with obvious definition of the~$\psi_i$,
so that $\psi(\delta_K)= \sum_{i=1}^r \psi(\lambda_i) \cdot(1- \psi_i(\sigma_i))
= \sum_{i=1}^r \psi (\lambda_i) \cdot(1- \xi_i)$, where the $\xi_i$ are
roots of unity of $p$-power order. So we can minimize the $p$-adic 
valuations of the $\psi(\delta_K)$ to obtain the best annihilator.

\smallskip
For instance, if $K$ is the compositum of two cyclic cubic fields and $p=3$,
whatever the choice of $\delta_K = \lambda_1\,(1-\sigma_1) + 
\lambda_2\,(1-\sigma_2)$, $\lambda_1, \lambda_2$ prime to $3$,
where $\sigma_1$, $\sigma_2$ are two generators of $G_K$, then 
$\psi(\delta_K) \sim 1-j$ for 6 characters and 
$\psi(\delta_K) \sim 3$ for 2 other characters $\psi \ne 1$.
So the result depends on the structures of the ${\mathcal T}_k$
of the 4 cubic subfields $k$ of $K$.

\begin{remarks} (i) Let $k$ be a subfield of $K$ and let $j_{K/k}$ 
be the ``transfer map'' ${\mathcal T}_k \to {\mathcal T}_K$. Then, for
$\delta_K {\mathcal A}_K$, we get:
$$(j_{K/k}({\mathcal T}_k))^ {\delta_K {\mathcal A}_K} = 
j_{K/k}({\mathcal T}_k^{{\rm N}_{K/k} (\delta_K {\mathcal A}_K)}) \simeq 
{\mathcal T}_k^{{\rm N}_{K/k} (\delta_K {\mathcal A}_K)} =
{\mathcal T}_k^{{\mathcal E}_k \cdot \delta_k {\mathcal A}_k}; $$
indeed, this comes from the injectivity 
of the transfer since the Leopoldt conjecture is true in abelian extensions 
(see e.g., \cite[Theorem IV.2.1]{Gr1}); then if the product of Euler factors 
${\mathcal E}_k := \prod_{\ell \mid f_K,\, \ell \nmid p f_k} 
\Big(1-\Frac{1}{\ell} \Big(\Frac{k}{\ell} \Big) \Big)$
is invertible (i.e., $\chi({\mathcal E}_k)$ prime to $p$ 
for all $\chi$), this means that there is no loss of information by using 
the annihilation of ${\mathcal T}_K$ by the $\delta_K {\mathcal A}_K$, instead of
that of ${\mathcal T}_k$ by the $\delta_k {\mathcal A}_k$; otherwise, it is not possible
to eliminate the Euler factors ``hidden'' in $\delta_K {\mathcal A}_K$ when they
are non-invertible (although they are never zero) unless to restrict ourselves 
to the use of the $\delta_k {\mathcal A}_k$ for ${\mathcal T}_k$, at the cost of a 
weaker information on the global Galois structure of ${\mathcal T}_K$.

\medskip
(ii) The $G_K$-module ${\mathcal T}_K$ gives rise to the
following submodules or quotients-modules which have interesting
arithmetical meaning and are of course annihilated by the 
$\delta_K {\mathcal A}_K$:\,%
\footnote{For some ${\mathcal C}_K := {\rm Gal}(H_K^*/K)$, 
$H_K^* \subseteq H_K^{\rm pr}$, we put
${\mathcal C}_K^\infty:= {\rm Gal}(K_\infty H_K^*/K_\infty)$.}

\smallskip
\quad $\bullet$ The submodule $\Cl_K^\infty := {\rm Gal}(K_\infty H_K/K_\infty)$ 
isomorphic to a sub-module of $\Cl_K$. Note that if $p$ is unramified in 
$K/\Q$ and if (for $p=2$) $-1$ is not a local norm at $2$, 
then $\Cl_K^\infty \simeq \Cl_K$ (cf. \eqref{clinfty}), which
explains that, in general, one says that the $p$-class group
is annihilated by the annihilators of ${\mathcal T}_K$.

\smallskip
\quad $\bullet$ The module ${\mathcal W}_K$ and the normalized $p$-adic regulator
${\mathcal R}_K$ defining the exact sequence \eqref{exseq}.

\smallskip
\quad $\bullet$ The Bertrandias--Payan module ${\mathcal B\mathcal P}_K := 
{\mathcal T}_K/{\mathcal W}_K$ for which the fixed field $H_K^{\rm bp}$
by ${\mathcal W}_K$ in $H_K^{\rm pr}/K_\infty$ is the compositum of the $p$-cyclic 
extensions of $K$ which are embeddable in $p$-cyclic extensions of arbitrary 
large degree.

\smallskip
Then some ``logarithmic objects'' defined and studied by Jaulent 
(see \cite{J1}, \cite[\S\,2.3, Sch\'ema]{J2} and \cite{B-J}),
in a theoretical and computational point of view:

\smallskip
\quad $\bullet$ The logarithmic class group $\wt \Cl_K := 
{\rm Gal}(H_K^{\rm lc}/K_\infty)$ ($H_K^{\rm lc}$ is the maximal 
abelian locally cyclotomic pro-$p$-extension of $K$), defining the
exact sequence
$1\to \wt \Cl_K^{[p]} \to \wt \Cl_K \to \Cl^{S \infty}_K  \to 1$
($\Cl^S_K := \Cl_K/\cl_K(S)$ is the $p$-group of $S$-classes of $K$
and $\wt \Cl_K^{[p]}$ the subgroup generated by $S$).

\smallskip
\quad $\bullet$ The ``logarithmic regulator'' $\wt {\mathcal R}_K$ as
quotient of the group of ``semi-local logarithmic units'' by the 
``global logarithmic units''.
\end{remarks}

\section{Conclusion}
This elementary study, especially with the help of numerical computations, 
shows that the broad generalizations of $\Z_p[G_K]$-annihilations, that come 
from values of partial $\zeta$-functions, with various base fields 
(see, e.g., \cite{Ng2, Ni, N-N, Sn} among many others), may be difficult
to analyse, owing to the fact that the results are not so efficients (especially 
in the non semi-simple and/or the non-cyclic cases), and that some 
degeneracies may occur because of Euler factors as soon as the $p$-adic 
pseudo-measures that are used are of ``Stickelberger's type`` like
Solomon's elements or cyclotomic units.

\smallskip
Moreover, Iwasawa's techniques give more elegant formalism 
but do not avoid the question of Euler factors.

\smallskip
Depending on whether one deals with imaginary or real fields, 
the suitable object to be annihilated is not defined in an unique way as 
shown by the context of the present paper about the $G_K$-module 
${\mathcal T}_K$.
Moreover, roughly speaking, some objects are relative to the values $L_p(0, \chi)$
(order of some component of the $p$-class group of some non-real ``mirror field''), 
while some other are relative to the values $L_p(1, \chi)$ 
(groups ${\mathcal T}_K$), and it is well known that the points
``$s=0$'' and ``$s=1$'' are mysteriousely independent, giving sometimes
abundant ``Siegel zeros'' near 1, as explained by Washington in many papers
(see \cite{Gr4} and its bibliography), whence an unpredictible order of magnitude
of the annihilators.

\end{document}